\documentclass[12pt]{amsart}
\usepackage{mathrsfs}
\usepackage{amsfonts, amssymb, amsmath, amscd}
\usepackage{enumerate}
\usepackage{graphicx}
\usepackage{tikz}
\usepackage{pgfplots}
\usetikzlibrary{shapes}
\usetikzlibrary{plotmarks}
\usepackage{stmaryrd}
\usepackage{booktabs}
\usepackage{mathtools}
\usepackage{young}
\usepackage{color,soul}
\usepackage{MnSymbol,wasysym}
\usepackage{calligra,mathrsfs}
\usetikzlibrary{matrix,arrows}

\usepackage[left=.9truein,right=.9truein,bottom=.9truein,top=.9truein]{geometry}

\makeatletter
\def\@tocline#1#2#3#4#5#6#7{\relax
  \ifnum #1>\c@tocdepth 
  \else
    \par \addpenalty\@secpenalty\addvspace{#2}%
    \begingroup \hyphenpenalty\@M
    \@ifempty{#4}{%
      \@tempdima\csname r@tocindent\number#1\endcsname\relax
    }{%
      \@tempdima#4\relax
    }%
    \parindent\z@ \leftskip#3\relax \advance\leftskip\@tempdima\relax
    \rightskip\@pnumwidth plus1em \parfillskip-\@pnumwidth
    #5\leavevmode\hskip-\@tempdima #6\relax
    \dotfill\hbox to\@pnumwidth{\@tocpagenum{#7}}\par
    \nobreak
    \endgroup
  \fi}
\makeatother

\let\oldtocsection=\tocsection
\let\oldtocsubsection=\tocsubsection

\renewcommand{\tocsection}[2]{\hspace{0em}\oldtocsection{#1}{#2}}
\renewcommand{\tocsubsection}[2]{\hspace{1.25em}\oldtocsubsection{#1}{#2}}

\usepackage{hyperref}

\newcommand*\rfrac[2]{{}^{#1}\!/_{#2}}

\newcommand{\ext}{\textnormal{ext}}
\newcommand{\Hom}{\textnormal{Hom}}

\newcommand{\sheafhom}{\mathscr{H}\text{\kern -3pt {\calligra\large om}}\,}
\newcommand{\sheafext}{\mathscr{E}\text{\kern -3pt {\calligra\large xt}}\,}

\newcommand{\Ext}{\textnormal{Ext}}

\newcommand{\ch}{\textnormal{ch}}

\newcommand{\ZZ}{\mathbb Z}
\newcommand{\QQ}{\mathbb Q}
\newcommand{\RR}{\mathbb R}
\newcommand{\CC}{\mathbb C}

\newcommand{\PP}{\mathbb P}

\newcommand{\OO}{\mathcal O}

\newcommand{\EEE}{\mathcal E}
\newcommand{\UU}{\mathcal U}
\newcommand{\LL}{\mathcal L}

\newcommand{\FFF}{\mathcal{F}}

\newcommand{\QS}{\mathbb{P}^1\times \mathbb{P}^1}
\newcommand{\gbar}{\overline{\gamma}}
\newcommand{\MMM}{M(\xi)}
\newcommand{\MS}{M^{s}(\xi)}

\newcommand{\osum}{\bigoplus}
\newcommand{\isom}{\cong}

\newcommand{\pic}{\textnormal{Pic}}

\newcommand{\ns}{\textnormal{NS}}

\newtheorem{theo}{Theorem}[section]
\newtheorem{theorem}[theo]{Theorem}
\newtheorem*{theorem*}{Theorem}
\newtheorem{thm}[theo]{Theorem}
\newtheorem*{thm*}{Theorem}
\newtheorem{proposition}[theo]{Proposition}
\newtheorem*{proposition*}{Proposition}
\newtheorem{prop}[theo]{Proposition}
\newtheorem*{prop*}{Proposition}

\newtheorem*{remark*}{Remark}
\newtheorem{lemma}[theo]{Lemma}
\newtheorem*{lemma*}{Lemma}
\newtheorem{cor}[theo]{Corollary}
\newtheorem*{cor*}{Corollary}

\newtheorem*{claim*}{Claim}

\newtheorem*{details*}{Details}

\newtheorem*{recall*}{Recall}
\newtheorem{ass}[theo]{Assumption}
\newtheorem*{ass*}{Assumption}
\newtheorem{conj}[theo]{Conjecture}
\newtheorem*{conj*}{Conjecture}

\newtheorem*{intprob*}{The Interpolation Problem}

\theoremstyle{definition}
\newtheorem{definition}[theo]{Definition}
\newtheorem*{definition*}{Definition}
\newtheorem{deff}[theo]{Definition}
\newtheorem*{deff*}{Definition}

\newtheorem*{problem*}{Problem}

\newtheorem*{prob*}{Problem}

\newcommand{\fto}[1]{\stackrel{#1}{\to}}
\begin{document}
\title{The Effective Cone of Moduli Spaces of Sheaves on a Smooth Quadric Surface}
\author{Tim Ryan}
\address{Dept. of Mathematics, Statistics, and Computer Science, University of Illinois at Chicago, Chicago, IL 60607}
\email{tryan8@uic.edu}
\maketitle
\begin{abstract}
Let $\xi$ be a stable Chern character on $\QS$, and let $M(\xi)$ be the moduli space of Gieseker semistable sheaves on $\QS$ with Chern character $\xi$.
In this paper, we provide an approach to computing the effective cone of $M(\xi)$.
We find Brill-Noether divisors spanning extremal rays of the effective cone using resolutions of the general elements of $M(\xi)$ which are found using the machinery of exceptional bundles.
We use this approach to provide many examples of extremal rays in these effective cones.
In particular, we completely compute the effective cone of the first fifteen Hilbert schemes of points on $\QS$.
\end{abstract}
\setcounter{tocdepth}{2}
\section{Introduction}
\label{sec:introduction}
In this paper, we provide an approach to computing extremal rays of the effective cone of moduli spaces of sheaves on $\QS$.
In particular, we show that this approach succeeds in computing the entire effective cone on the first fifteen Hilbert schemes of points.

The effective cone of a scheme is an important invariant which controls much of the geometry of the scheme \cite{L:L04}.
For Mori dream spaces, it determines all of the birational contractions of the space \cite{HK:HK00}.
However, in general, determining the effective cone of a scheme is a very difficult question.
There has been progress computing the effective cone for certain moduli spaces. 

Moduli spaces of sheaves (on a fixed surface) are one kind of moduli space that has been extensively studied (e.g. \cite{BCZ:BCZ16}, \cite{BHLRSWZ:BHLRSWZTBD}, \cite{CC:CC15}, \cite{CH4:CH15}, \cite{DN:DN89}, \cite{Fo:Fo73}, \cite{Gi:Gi77}, \cite{LQ:LQ94}, \cite{Ma:Ma77}, \cite{Ma3:Ma75}, \cite{MO:MO07}, \cite{MO2:MO08}, \cite{MW:MW97}, \cite{Ta:Ta72}, \cite{Y4:Y12}).
In this setting, the geometry of the underlying variety can be used to study the moduli space.
In the past decade, 
Bridgeland stability has motivated a program to compute the effective cones of these moduli spaces by corresponding the edge of the effective cone with the \textit{collapsing wall} of Bridgeland stability.
The recent advances in Bridgeland stabilty (e.g. \cite{AB:AB13}, \cite{B:B07}, \cite{B2:B208}, \cite{BBMT:BBMT14}, \cite{BeMa:BeMa15}, \cite{BM:BM14}, \cite{BMS:BMSTBD}, \cite{BMT:BMT14}, \cite{CHP:CHP16}, \cite{LoQi:LoQi14}, \cite{Mac:Mac14}, \cite{Maci2:Maci}, \cite{MM:MM13}, \cite{MP:MP13}, \cite{MS:MS16}, \cite{Nu2:Nu14:}, \cite{S:S14}, \cite{To:To14}, \cite{To2:To14}, \cite{To3:To14-3}, \cite{Tr:Tr15}, \cite{Y4:Y12}, \cite{Y5:Y14}) have helped this approach be successful in general on \textit{K3 surfaces} \cite{BM2:BM14-2}, \textit{Enriques surfaces} \cite{Nu:NuTBD}, \textit{Abelian surfaces} \cite{Y2:Y01}, and $\PP^2$ \cite{CHW:CHW14}.

The proof in the last case varies greatly from the proofs in the other cases as it is a surface of negative Kodaira dimension.
More generally, there has been a lot of work on $\PP^2$ (e.g. \cite{ABCH:ABCH13}, \cite{BMW:BMW14}, \cite{CH:CHTBD}, \cite{CH2:CH14}, \cite{CH3:CH14-2}, \cite{DLP:DLP85}, \cite{Hu:Hu16}, \cite{Hu2:Hu13}, \cite{LZ:LZ13}, \cite{LZ2:LZTBD}, \cite{Wo:Wo13}). 
Although much of this work has been extended to more rational and ruled surfaces (e.g. \cite{A:ATBD}, \cite{Ba:Ba87}, \cite{Go:Go96}, \cite{K2:K94}, \cite{Mo:Mo13}, \cite{Q:Q92}), no general method to compute the entire effective cone of a moduli space of sheaves on $\QS$ has been given.
This is because the proof in \cite{CHW:CHW14} relies heavily on properties that are unique to $\PP^2$.

This paper provides the general framework to potentially extend the results of \cite{CHW:CHW14} to del Pezzo and Hirzebruch surfaces and explicitly works out the framework on $\QS$.
Under some additional hypotheses, this framework gives a method to compute the entire effective cone of moduli spaces of sheaves on $\QS$.
The increased ranks of the Picard group and the derived category in the case of $\QS$ compared to $\PP^2$ make the proofs and results significantly harder to obtain.

These difficulties force us to add two new ingredients to the method.
The first new addition is putting the choice of an exceptional collection in the context of the work of Rudakov et al. on coils and helices (e.g. \cite{R:R89}, \cite{NZ:NZ90}, \cite{G:G89}).
Certain special properties of exceptional collections on $\PP^2$ that were used are no longer needed once the choice is put in these terms.
The second addition is providing a way to link neighboring extremal rays to show that there are no missing extremal rays in between them.
This addition is needed as the Picard rank of the moduli spaces are now higher than two, and it will be essential for expanding these results to other surfaces.

Let $\xi$ be a Chern character of postive integer rank on $\QS$.
Then there is a nonempty moduli space $M(\xi)$ that parametrizes $S$-equivalence classes of semistable sheaves with that Chern character on $\QS$ iff $\xi$ satisfies a set of Bogomolov type inequalities given by Rudakov in \cite{R2:R94}.
It is an irreducible \cite{Wa:Wa98}, normal \cite{Wa:Wa98}, projective variety \cite{Ma2:Ma78}.
We show that these spaces are $\QQ$-factorial [Prop. \ref{prop:Q-factorial}] and, furthermore, are Mori dream spaces [Thm. \ref{thm:MDS}].

We construct effective Brill-Noether divisors of the form
$$D_V = \{U \in \MMM : h^1(U \otimes V) \neq 0\}.$$
We create an algorithmic method to produce these divisors.
Conjecturally, this method produces a set of divisors spanning the effective cone.

\begin{conj*}
The method laid out in this paper produces a set of effective divisors spanning the effective cone for $M(\xi)$ for all $\xi$ above Rudakov's surface.
\end{conj*}

One reason for this conjecture is that the method computes the entire effective cone of the first fifteen Hilbert schemes of points on $\QS$ (which is as many as we applied it to).
In the last section of the paper, we explicitly compute the effective cone of $\left(\QS\right)^{\lbrack n\rbrack}$ for $n \leq 16$ as well as several instances of types of extremal rays that show up in infinite sequences of $n$, and we give a rank two example.
Even if this method fails to fully compute the effective cone of every moduli space, it does give a method to produce effective divisors on these moduli spaces.

The proofs of the steps of the method follow from constructing birational maps or Fano fibrations to simpler, Picard rank one, spaces and analyzing them to find our extremal divisors.
Given a birational map or Fano fibration, giving an divisor on an edge of the effective cone follows directly.
The difficult part of the process is constructing the maps.

The moduli spaces that we map to are moduli spaces of Kronecker modules, $Kr_V(m,n)$.
The way we construct the map is to find a resolution of the general element of $M(\xi)$ containing a Kronecker module and then to forget the rest of the resolution.
The key then is to find resolutions of the general element of the moduli space that contain Kronecker modules.
On $\QS$, we have a powerful tool for finding resolutions of sheaves in the form of a generalized Beilinson spectral sequence \cite{G:G89}.
Using this method for finding a resolution will make it clear that the extremal divisors we construct are Brill-Noether divisors as they will be defined in terms of the jumping of ranks of cohomology groups appearing in the spectral sequence.
Using that spectral sequence, finding resolutions with Kronecker modules is reduced to finding the right collections of exceptional bundles spanning the derived category.

We find the elements of the right collections by studying Rudakov's classification of stable bundles over the hyperbola of Chern characters $\zeta$ with the properties 
$$\chi(\zeta^*,\xi) = 0 \textnormal{ and } \Delta(\zeta) = \frac{1}{2}.$$ 
Rudakov's necessary and sufficient inequalities for a Chern character to be stable each came from an exceptional bundle on $\QS$, and the right collections of exceptional bundles are determined by which exceptional bundles have the sharpest inequalities over this curve.
Say that $(E_\alpha,E_\beta)$ is an exceptional pair of bundles that have the sharpest inequalities over that curve.
Then the resolution we get for the general $U \in \MMM$ might look like
$$0 \to E_\alpha^*(K)^{m_3} \osum F_{0}^{*m_2} \to F_{-1}^{*m_1} \osum E_\beta^{*m_0} \to U \to 0.$$

Using this resolution, we get the maps we need.
There are several cases to be dealt with.
On $\PP^2$, there was only two cases.
The new cases are a phenomenon that will persist on other surfaces and are not unique to $\QS$.

We summarize the approach in the most common case.
The resolution of the general object of the moduli space in this case will look like the example resolution above.
Then the resolution has four objects and four maps.
The map $W: F_{0}^{*m_2} \to F_{-1}^{*m_1}$ gives the required Kronecker module. 
We map to the Kronecker moduli space corresponding to it,  $f: \MMM \dashrightarrow Kr_V(m,n)$.
Constructing the Brill-Noether divisor in this case is slightly tricky because the bundle whose corresponding divisor spans the extremal ray is not obviously cohomologically orthogonal to the general object of the moduli space.
That orthogonality is established using properties of the Kronecker modules in the resolution of a bundle whose corresponding divisor spans the extremal ray and in the resolution of the general object.

\subsection{Organization}
In Sec. \ref{sec:prelim}, we extensively lay out the necessary background and prove two properties of the moduli spaces we want to study.
In Sec. \ref{sec:excpair}, we define primary orthogonal Chern characters via controlling exceptional pairs.
In Sec. \ref{sec: Beilinson Spec seq}, we use the controlling pairs and a generalized Beilinson spectral sequence to resolve the general object of our moduli space which constructs effective divisors on our space.
In Sec. \ref{sec: The Kron. Fib.}, we use these resolutions to construct maps from our moduli space to spaces of Kronecker modules that provide the dual moving curves we need.
In Sec. \ref{sec: primary extremal rays of the effective cone}, we use these results to compute extremal rays of the effective cone of $M(\xi)$.
Finally, in Sec. \ref{sec: examples}, we compute the effective cone of 
for $n\leq 16$, provide some recurring examples of types of corners, and work out a rank two example.

\section{Preliminaries}
\label{sec:prelim}
In this section, we will discuss all preliminary material needed to understand the following sections. 
We base our discussion of the general preliminaries on sections in \cite{HuL:HuL10} and \cite{LP:LP97}.
For the subsections more specific to $\QS$, we also draw on \cite{R:R89} and \cite{R2:R94}.

In this paper, all sheaves will be coherent torsion free sheaves.
Other than Prop. \ref{prop:Q-factorial} and Thm. \ref{thm:MDS}, we will work exclusively on $\QS$ and so will drop that from labels as convenient.
Let $E$ be a sheaf with Chern character $\xi = \left( \ch_0, \ch_1, \ch_2 \right)$.


\subsection{Slopes and the discriminant}
\label{subsec:slopes and the discriminant}
Recall that for a locally free sheaf, we have $$\ch_0(E) = \textnormal{rank}(E) = r(E).$$

Using this equality, we define the \textit{slope} and \textit{discriminant} of $E$ to be 
$$\mu(E) = \frac{\ch_1(E)}{\ch_0(E)} \textnormal{ and } \Delta(E) = \frac{1}{2}\mu(E)^2-\frac{\ch _2(E)}{\ch_0(E)}.$$
The slope and discriminant are more convenient than the Chern character for us as they have the property that 
$$\mu(E \otimes F) = \mu(E) +\mu(F) \textnormal{ and } \Delta(E \otimes F) = \Delta(E) +\Delta(F).$$ 
We should note that these notions are easily extended to Chern characters in $K\left(\QS\right) \otimes \RR$.

On a Picard rank one variety, the slope is a generalization of the degree of a line bundle to higher rank vector bundles.
On higher Picard rank varieties, we can think of the slope as a generalization of a multi-degree that carries the information of the degree for every choice of embedding.
Sometimes we would like an analog to the degree with respect to a specific embedding so we give the following definition.
The \textit{slope with respect to an ample divisor $H$} is $$\mu_H(E) = \mu(E) \cdot H.$$
We also call this the \textit{$H$-Slope} of $E$.
The two ample divisors that we will use in this paper are $\OO_{\QS}(1,1)$ and $\OO_{\QS}(1,2)$ for which the slopes will be denoted $\mu_{1,1}$ and $\mu_{1,2}$, respectively.

We want to write the Riemann-Roch formula in terms of these invariants so let's recall the definition of the (reduced) Hilbert polynomial.

\begin{deff}
\label{def:Hilb Poly}
The \textit{Hilbert polynomial of a sheaf $\FFF$ with respect to an ample line bundle $H$} is $$P_\FFF(k) = \chi\left( \FFF(k) \right) = \frac{\alpha_d}{d!}k^d+\cdots+\alpha_1 k+\alpha_0$$ where $\FFF(k) = \FFF \otimes H^{\otimes k}$, where the dimension of $\FFF$ is $d$, and where we think of this as a polynomial in the variable $k$.
\end{deff}

\begin{deff}
\label{def:reduced hilb poly}
The \textit{reduced Hilbert polynomial} of a sheaf $\FFF$ with respect to an ample line bundle $H$ is defined to be $$p_\FFF(k) = \frac{P_\FFF(k)}{\alpha_d}$$
where $d$ is the dimension of $\FFF$.
\end{deff}
The Hilbert polynomial, unlike the individual cohomology groups it sums over, is a numerical object that is entirely determined by the Chern character of the sheaf and the Chern character of the ample line bundle.

When $\ch_0(E)>0$, we can write the Riemann-Roch formula as
$$\chi(E)=r\left( P(\mu)-\Delta \right),$$
where $P(\mu)$ is the Hilbert polynomial of $\OO$, 
$$P(m,n) = \frac{1}{2} (m,n)^2 +(1,1) \cdot (m,n) +1 = mn+m+n+1 =(m+1)(n+1).$$

\subsubsection{Classical Stability Conditions}
\label{subsubsec: Classical Stability Conditions}

Using our notations of slope and the reduced Hilbert polynomial,
we can now define the different classical notions of stability that will we need.

\begin{deff}
\label{def:slope stability}
A sheaf $\FFF$ is \textit{slope (semi-)stable} with respect to an ample line bundle $H$ if for all proper subsheaves $\FFF' \subset \FFF$, $$\mu_H(\FFF') \left( \leq \right) < \mu_H(\FFF) .$$ 
\end{deff}

A stronger notation of stability is the notion of \textit{Gieseker stability}, which is also known as $\gamma$ stability.

\begin{deff}
A sheaf $\FFF$ is \textit{Gieseker (semi-)stable} with respect to an ample divisor $H$ if for all proper subsheaves $\FFF' \subset \FFF$, $p_{\FFF'} \left( \leq \right) < p_\FFF$ where the polynomials are compared for all sufficiently large input values.
We will also call this $\gamma$ (semi-)stability.
\end{deff}
The ordering on the polynomials could also have been phrased as the lexiographic ordering on their coefficients starting from the highest degree term's coefficient and working down.
The condition on the Hilbert polynomial is equivalent on surfaces to requiring $\mu_H(\FFF') \leq \mu_H(\FFF)$ and, in the case of equality, $\Delta(\FFF') \left(\geq\right) > \Delta(\FFF)$.

These two notations of stability and semi-stability are related by a string of implications that seems slightly odd at first, but becomes clear using this last equivalence: $$\textnormal{slope stable} \rightarrow \textnormal{Gieseker stable} \rightarrow \textnormal{Gieseker semi-stable} \rightarrow \textnormal{slope semi-stable}.$$

As we will be focused on $\QS$, one additional notion of stability will be relevant.
\begin{deff}
A sheaf $\FFF$ is \textit{$\overline{\gamma}$ (semi-)stable} if it is Gieseker semi-stable with respect to $\OO(1,1)$ and for all proper subsheaves $\FFF' \subset \FFF$, if $\mu_{1,1}\left(\FFF'\right) = \mu_{1,1}\left(\FFF\right)$ and  $\Delta\left(\FFF'\right) = \Delta\left(\FFF\right)$, then $\mu_{1,2}\left(\FFF'\right) \left(\leq \right)< \mu_{1,2}\left(\FFF\right)$.
\end{deff}

Again we have implications:  $$\textnormal{slope stable} \rightarrow \textnormal{Gieseker stable} \rightarrow \overline{\gamma}\textnormal{ stable}$$ 
$$\rightarrow \overline{\gamma}\textnormal{ semi-stable} \rightarrow \textnormal{Gieseker semi-stable} \rightarrow \textnormal{slope semi-stable}.$$
Because Gieseker stability generalizes to all varieties, it might seem odd to add a third, very variety specific, condition.
However, adding this third condition allows $\overline{\gamma}$ stability to order all Chern characters in a way that was not possible for slope or Gieseker stability.
This order is possible with $\overline{\gamma}$ stability because it has three conditions that can distinguish the three variables for a Chern character on $\QS$: the two coordinates of $c_1$ and the coordinate of $c_2$.



\subsection{Exceptional Bundles}
\label{subsec: exceptional bundles}
In this subsection, we start with a relative version of Euler characteristic.
We will eventually see that sheaves $E$ with relative Euler characteristic $\chi(E,E) = 1$ will control the geometry of our moduli spaces. 

\subsubsection{The Relative Euler Characteristic}
\label{subsubsec: rel Euler char}

The \textit{relative Euler characteristic} of two sheaves on a variety $X$ of dimension $n$ is
$$\chi(E,F) = \sum_{i=0}^n (-1)^i \textnormal{ext}^i(E,F)$$
where $\textnormal{ext}^i(E,F) = \dim(\Ext^i(E,F))$.
For locally free sheaves, we can equivalently define it by the formula
$$\chi(E,F) = \chi(E^*\otimes F) = \sum_{i=0}^n (-1)^i h^i(E^*\otimes F).$$
Restricting to the case where $n=2$, we again write out Hirzebruch-Riemann-Roch in order to explicitly compute this as 
$$\chi(E,F) = r(E) r(F) \left( P\left( \mu(E) -\mu(F) \right) -\Delta(E)- \Delta(F) \right)$$
where $P(x,y) = (x+1)(y+1)$ is the Euler characteristic of  $\OO_{\QS}(x,y)$.

Using the alternate definition of relative Euler characteristic, we define a bilinear pairing on \\$K\left( \QS\right)$ by
$$\left( E,F\right) = \chi(E^*,F) = \chi(F^*,E) = \chi(E \otimes F).$$
Then we define the \textit{orthogonal complement} of $E$, denoted $\ch(E)^\perp$, to be all bundles $F$ such that $(E,F) = 0$.
Note that the pairing is symmetric 
so the orthogonal complement does not depend on whether $E$ is the first or second element in the pairing.


The $\Ext$ groups we used to define the relative Euler characteristic also allow us to describe a set of vector bundles that control the geometry of $\QS$.
\begin{deff}A sheaf $E$ is \textit{exceptional} if $\Hom(E,E) = \CC$ and $\Ext^i(E,E)=0$ for all $i>0$.\end{deff}
We say a Chern character (slope) is \textit{exceptional} if there is an exceptional sheaf with that character (slope).
The prototypical exceptional sheaves are the line bundles, $\OO(a,b)$.
Many of the properties of line bundles pass to exceptional sheaves on $\QS$.
We list some of these properties without repeating the proofs.
Their proofs are in the previously mentioned paper by Gorodentsev \cite{G:G89}.
\begin{prop}On $\QS$, we have the following results.
\leavevmode
\begin{list}{$\bullet$}{}
\item Every exceptional sheaf is a vector bundle.
\item There exists a unique exceptional bundle for each exceptional slope.
\item The two coordinates of the first Chern class are coprime with the rank of $E$.
\item The denominator of $\mu_{1,1}(E)$ is the rank of $E$.
\item Exceptional bundles are stable with respect to $\OO(1,1)$.
\end{list}
\end{prop}

Given these properties, it makes sense to define the \textit{rank} of an exceptional slope $\nu$ to be the smallest $r \in \ZZ$ such that $r \mu_{1,1}(\nu) \in \ZZ$.
Also, set the notation $E_{\frac{a}{r},\frac{b}{r}}$ for the unique exceptional bundle with slope $\left(\frac{a}{r},\frac{b}{r} \right)$.
Given a Chern character $\xi$ with slope $\alpha = \left( \frac{a}{r},\frac{b}{r} \right)$, we will interchangeably write $E_{-\alpha} = E_\alpha^* = E_{\frac{-a}{r},\frac{-b}{r}}$.
Similarly, we will abuse notation and write the slope and the whole Chern character interchangeably.

We can characterize exceptional bundles among all stable bundles by the Euler characteristic $\chi(E,E)$. 
For stable $E$, Serre duality implies $\Ext^2(E,E) = \Hom(E,E(K)) = 0$.
Similarly, the stability of $E$ implies that $\Hom(E,E) = \CC$.
Then for stable $E$ we have
$$\chi(E,E)= 1 -\textnormal{ext}^1(E,E) \leq 1$$
with equality precisely when $E$ is exceptional.
Conversely, we use Hirzebruch-Riemann-Roch to explicitly compute 
$$\chi(E,E) = r(E)^2\left(1-2\Delta \right).$$
Putting these together, we see that for a stable bundle
$$\Delta(E) = \frac{1}{2}\left( 1-\frac{\chi(E,E)}{r(E)^2}\right).$$
For exceptional bundles, this reduces to
$$\Delta(E) = \frac{1}{2}\left( 1-\frac{1}{r(E)^2}\right).$$
Since $\chi(F,F)\leq 0$ for all other stable bundles, we see that exceptional bundles are the only stable bundles with $\Delta < \frac{1}{2}$.
As there is a unique exceptional bundle for an exceptional slope and there can be no other stable bundles with that discriminant, the moduli space of stable bundles with an exceptional Chern character is a single reduced point \cite{G:G89}.
Giving an explicit description of the exceptional bundles analogous to the description of the exceptional bundles on $\PP^2$ given in \cite{CHW:CHW14} and \cite{LP:LP97} is an open question.


\subsection{Exceptional Collections}
\label{subsec: exceptional collections}

Exceptional bundles naturally sit inside of collections.
\begin{deff}A collection of exceptional bundles $\left( E_0,\cdots,E_n\right)$ is an \textit{exceptional collection} if for all $i<j$, $\Ext^k(E_j,E_i) = 0$  for all $k$ and there is at most one $k$ such that $\Ext^k(E_i,E_j) \neq 0$.\end{deff}
An exceptional collection is \textit{strong} if $k=0$ for all pairs $(i,j)$.
We say the \textit{length} an exceptional collection $\left( E_0,\cdots,E_n\right)$ is $n+1$.
An \textit{exceptional pair} is an exceptional collection of length two. 
A \textit{coil} is a maximal exceptional collection, which in the case of $\QS$ is length four.
Our stereotypical (strong) coil is $\left( \OO,\OO(1,0),\OO(0,1),\OO(1,1)\right)$.

Every exceptional bundle sits inside of an exceptional collection, and every exceptional collection can be completed to a coil \cite{R2:R94}.
Given an exceptional collection of length three $(E_0,E_1,E_2)$ there are exactly four ways to complete it to a coil:
$$(A,E_0,E_1,E_2), (E_0,B,E_1,E_2), (E_0,E_1,C,E_2) \textnormal{, and } (E_0,E_1,E_2,D).$$
In other words, once you pick where you would like the fourth bundle to be, there is a unique way to complete the exceptional collection to a coil.
This uniqueness follows as an easy consequence of the requirement that the fourth bundles forms an exceptional pair in the correct way with the other three bundles.
First, each bundle we require it to be in an exceptional pair with imposes an independent condition on its rank and first Chern classes so they are determined.
Then the rank and first Chern class determine its discriminant, and the bundle is uniquely determined by its Chern classes.

Before we can state how to extend an exceptional collection of length two to a coil, we first must explain a process that turns an exceptional collection into different exceptional collections called \textit{mutation} or \textit{reconstruction}.
We first define mutation for exceptional pairs and then bootstrap this definition into a definition for all exceptional collections.
The definitions of mutation that we use are equivalent on $\QS$ to the general definitions \cite{G:G89}.
\begin{deff}The \textit{left mutation of an exceptional pair} $\left(E_0,E_1\right)$ is the exceptional pair $L\left(E_0,E_1\right) = \left( L_{E_0}E_1,E_0\right)$ where $L_{E_0}E_1$ is determined by one of the following short exact sequences:
$$\textnormal{(regular) } 0\to L_{E_0}E_1 \to E_0 \otimes \Hom(E_0,E_1) \to E_1 \to 0,$$
$$\textnormal{(rebound) } 0\to E_0 \otimes \Hom(E_0,E_1) \to E_1 \to L_{E_0}E_1 \to 0\textnormal{, or}$$
$$\textnormal{(extension) } 0\to E_1 \to L_{E_0}E_1 \to E_0 \otimes \Ext^1(E_0,E_1) \to 0.$$\end{deff}
One of these sequences exists by Gorodentsev \cite{G:G89}.
By rank considerations only one of the previous short exact sequences is possible so the left mutation is unique.
Rebound and extension mutations are called \textit{non-regular}.
\textit{Right mutation of an exceptional pair}, denoted $R\left(E_0,E_1\right)=\left( E_1,R_{E_1}E_0 \right)$, is defined similarly by tensoring the $\Hom$ or $\Ext$ with $E_1$ rather than with $E_0$.
Note that left and right mutation are inverse operations in the sense that 
$$L\left( R\left(E_0,E_1\right) \right) = R\left( L\left(E_0,E_1\right) \right) =\left(E_0,E_1\right).$$

We can also mutate any part of an exceptional collection.
In particular, replacing any adjacent exceptional pair in an exceptional collection with any of its left or right mutations gives another exceptional collection.
For example, given an exceptional collection $\left( E_0,E_1,E_2,E_3\right)$,
$$\left(  L(E_0,E_1),E_2,E_3 \right), \left(  R(E_0,E_1),E_2,E_3 \right), \left(  E_0,L(E_1,E_2),E_3 \right)$$
$$\left(  E_0,R(E_1,E_2),E_3 \right), \left(  E_0,E_1,L(E_2,E_3) \right)\textnormal{, and } \left(  E_0,E_1,R(E_2,E_3) \right)$$
are all exceptional collections. 
Rudakov proved that all possible exceptional collections can be gotten from our stereotypical collection via these pairwise reconstructions \cite{R:R89}. 

Mutating an exceptional collection is then just mutating all of the bundles in a systematic way.
We define the \textit{left(right) mutation of an exceptional collection} $\left( E_0,E_1,E_2,,\cdots, E_n\right)$ as 
$$\left( L_{E_0}\cdots L_{E_{n-1}} E_n, \cdots, L_{E_0}L_{E_1}E_2 ,L_{E_0}E_1,E_0\right)$$
$$\left(\left( E_n,R_{E_n}E_{n-1} ,R_{E_n}R_{E_{n-1}}E_{n-2},\cdots, R_{E_n}\cdots R_{E_1}E_0\right)\right).$$
For a coil $(E_0,E_1,F_0,F_1)$ on $\QS$, its left mutation is $(F_{1}(K),F_2(K),E_{-1},E_0)$ and its right mutation $(F_1,F_2,E_{-1}(-K),E_0(-K))$.

Now, let's return to the problem of completing an exceptional pair to a coil.
Say that we could extend $(E_0,E_1)$ to the coil $(E_0,E_1,F_0,F_1)$, this extension is not unique, even up to placement because $(E_0,E_1,L(F_0,F_1))$ and $(E_0,E_1,R(F_0,F_1))$ are also coils.
To make a unique notation of extension, we need the notion of a \textit{system}.
\begin{deff}Using mutation, each exceptional pair $\left(E_0,E_1\right)$ generates a \textit{system of exceptional bundles} $\{ E_i\}_{i\in\ZZ}$ where we inductively define $E_{i+1} = R_{E_i}E_{i-1}$ and $E_{i-1} = L_{E_i}E_{i+1}$.\end{deff}
Given an exceptional pair  $\left(E_0,E_1\right)$, we define the \textit{right completion system} of it to be the unique system $\{F_i\}$ such that  $\left(E_0,E_1,F_i,F_{i+1}\right)$ is a coil.
\textit{Left} and \textit{center} completion systems are defined analogously.
A \textit{left (right,center) completion pair} is any pair $\left(F_i,F_{i+1}\right)$ coming from the left (right,center) completion system.
By Prop. 4.5 \& 4.8 of Rudakov's paper \cite{R:R89}, the completion system is either a system of line bundles or has a \textit{minimally ranked completion pair} where the minimally ranked completion pair is the pair in the system with the lowest sum of the two ranks of the bundles in the system.

Completing an exceptional collection of length one to a coil can be reduced to the two steps of completing it to an exceptional pair (which is highly not unique) and then completing the pair to a coil.
In this paper, we will start with pairs and provide a unique way to extend them to a coil.

Given a complex $W:A^a  \to B^b$ of powers of an exceptional pair $(A,B)$, we extend the idea of mutation to the complex.
Define $L W$ to be the complex 
$$LW: (L_A B)^a \to A^b$$
and similarly define $R W$ to be the complex 
$$RW: B^a \to (R_B A)^b.$$

We also define the mutations relative to an exceptional bundle $C$ where $\{C,A,B\} \left(\{A,B,C\}\right)$ is an exceptional collection as follows:
If $\{C,A,B\}$ is an exceptional collection, define $L_C W$ to be the complex 
$$L_C W: (L_C A)^a \to (L_C B)^b,$$
and similarly, if $\{A,B,C\}$ is an exceptional collection, define $R_C W$ to be the complex 
$$R_C W: (R_C A)^a \to (R_C B)^b.$$

We can now state a theorem of Gorodentsev which establishes the spectral sequence that we need on $\QS$.
\begin{thm}[\cite{G:G89}]
\label{thm:specseq}
Let $U$ be a coherent sheaf on $\QS$, and let $\left( E_0,E_1,F_0,F_1 \right)$ be a coil.
Write $\mathcal{A} = \left( A_{0},A_{1},A_{2},A_{3}\right) = \left(E_0,E_1,F_0,F_1 \right)$ and $\mathcal{B} = \left( B_{-3},B_{-2},B_{-1},B_{0} \right) =\left( F_{1}(K), F_{2}(K), E_{-1},E_0\right)$
There is a spectral sequence with $E_1^{p,q}$-page
$$E_1^{p,q} = B_p \otimes \Ext^{q-\Delta_p}\left(A_{-p} ,U\right)$$
that converges to $U$ in degree 0 and to 0 in all other degrees
where $\Delta_p$ is the number of non-regular mutations in the string $L_0 ...L_{p-1} A_p$ which mutates $A_p$ into $B_{-p}$.
\end{thm}

It should be clear that $\Delta_0=0$. 
Considering the spectral sequence converging to different bundles of the coil allows us to deduce that $\Delta_3 = 1$ and that the other two are either 0 or 1 [Rmk. 1.5.2, \cite{K2:K94}].
Also, notice that $\mathcal{B}$ is the left mutation of $\mathcal{A}$.

\subsection{Moduli spaces of sheaves}

Let $E$ be exceptional with Chern character $e$.
Then, $E$ imposes a numerical condition on stable bundles with ``nearby'' Chern characters.
To see this condition, start with an exceptional bundle $E$ and another stable bundle $F$, we know that there are no maps from $E$ to $F$ if the $(1,1)$-slope of $E$ is bigger than $F$'s by $F$'s stability so
$$\hom(E,F) = 0.$$
Similarly, there are no maps from $F$ to $E$ twisted by the canonical if the $(1,1)$-slope of $F$ is greater than the $(1,1)$-slope of $E(K)$ so
$$\hom(F,E(K)) = 0.$$
By Serre duality, 
$$\hom(F,E(K)) =\ext^2(E,F) = 0.$$
Thus, if we have both of these conditions, we know that 
$$\chi(E,F) = \textnormal{ext}^1(E,F) \leq 0.$$
This is the numeric condition that $E$ imposes on nearby stable bundles.
Given a fixed exceptional $E$, we can encode this data by saying that the Chern character of $F$ must lie on or above a certain surface, $\delta_E$, in the $(\mu,\Delta)$ space.
We define $\delta_E(\mu)$
\begin{displaymath}
   \delta_E(\mu) = \left\{
     \begin{array}{lr}
       \chi(E,\mu)=0 & \textnormal{ if } \mu_{1,1}(E)-4 <\mu_{1,1}(\mu) \textnormal{ and } \gbar(\mu)<\gbar(E) \\
       \chi(\mu,E)=0 & \textnormal{ if } \mu_{1,1}(\mu) <\mu_{1,1}(E) +4  \textnormal{ and } \gbar(E)<\gbar(\mu) \\
       0 & \textnormal{otherwise                   }
     \end{array}
   \right\}.
\end{displaymath}
Then $F$'s Chern character lying on or above $\delta_E$ means that 
$$\Delta(F) \geq \delta_E(\mu(F)).$$

Using these conditions, each exceptional bundle gives an inequality that a Chern character must satisfy in order to be stable.
We combine all of these conditions into one by looking at the maximum over all of the inequalities.
Formally, let $\EEE$ be the set of exceptional bundles and define the $\delta$ surface by  
\begin{definition}\label{def:deltasurface}
$$\delta(\mu) = \sup_{\{E\in \EEE \}} \delta_E\left( \mu \right).$$
\end{definition}
Then saying that a stable Chern character, $\zeta$, must satisfy all of the inequalities from exceptional bundles is equivalent to
$$\Delta(F) \geq \delta(\mu(F)).$$
Alternatively, we say that a stable Chern character must lie on or above the $\delta$-surface.

Rudakov proved that lying above the $\delta$ surface was not only necessary but also sufficient for a Chern character to be stable.
\begin{theorem}\label{thm:modulispace}[Main theorem, \cite{R2:R94}]
Let $\xi = (r,\mu,\Delta)$ be a Chern character of postive integer rank. There exists a positive dimensional moduli space of $\gbar$-semistable sheaves $M_{\gbar}\left( \xi \right)$ with Chern character $\xi$ if and only if $c_1(\mu) \cdot(1,1)=r\mu_{1,1} \in \ZZ$, $\chi = r\left( P(\mu)-\Delta\right) \in \ZZ$, and $\Delta \geq \delta(\mu)$\cite{R2:R94}.
The same conditions are necessary and sufficient for $\gamma$-semistability as long as $\Delta > \frac{1}{2}$ and $\mu \notin \EEE \ZZ$\cite{R2:R94}.
\end{theorem}


\subsection{New basic properties}
\label{subsec:new prop}
With this background, we can begin our original work on $\MMM$.
In the last subsection, we listed many properties of the nonempty moduli spaces. 
In this subsection, we prove that the moduli spaces are $\QQ$-factorial and are Mori dream spaces.
In fact, we prove these results for all moduli spaces of Gieseker semistable sheaves on any del Pezzo surface, not just $\QS$.

\begin{prop}
\label{prop:Q-factorial}
Let $M$ be the moduli space of semistable sheaves with a fixed Chern character on a del Pezzo surface.
Then $M$ is $\QQ$-factorial.
\end{prop}

\begin{proof}
By \cite{D2:D91}, $M$ is a geometric quotient of a smooth variety.
Applying Thm. 4 of \cite{Hau:Hau01} immediately allows us to conclude that $M$ is $\QQ$-factorial. 
\end{proof}

The proof that these spaces are Mori dream spaces is slightly more involved, but similar in flavor.
\begin{thm}
\label{thm:MDS}
Let $M$ be the moduli space of semistable sheaves with a fixed Chern character on a del Pezzo surface.
Then $M$ is a Mori dream space.
\end{thm}

\begin{proof}
The proof follows the same basic outline as the proof for moduli spaces of sheaves on $\PP^2$ \cite{CHW:CHW14}.
By \cite{BCHM:BCHM10}, a log Fano variety is a Mori dream space.
This result reduces the theorem to showing that $M$ is a log Fano variety.
Since the anticanonical bundle of $M$ is nef [\cite{HuL:HuL10},Thm. 8.2.8 \& 8.3.3] and there exists effective divisors $E$ such that $-K_{M}-\epsilon E$ is ample for all sufficiently small $\epsilon >0$, showing $M$ is log Fano reduces to showing that $\left( M,\epsilon E\right)$ is a klt-pair for all effective divisors $E$.
Showing that $\left( M,\epsilon E\right)$ is a klt-pair for all effective divisors $E$ further reduces to showing that $M$ has canonical singularities. 

We now show that $M$ has canonical singularities. 
By \cite{D2:D91}, $M$ is also a geometric quotient of a smooth variety.
By \cite{Bo:Bo87}, a geometric quotient of a variety with rational singularities has rational singularities so $M$ has rational singularities.
As a $M$ is 1-Gorenstein [Thm. 8.3.3, \cite{HuL:HuL10}], it has canonical singularities [Thm. 11.1, \cite{K3:K96}]. 
\end{proof}


\subsection{Additional Assumptions}
\label{subsec: additional assumptions}
There are two previously mentioned properties that we would also like our moduli spaces on $\QS$ to have. 
We want the complement of the stable locus to be codimension at least two, and we want them to have Picard rank 3.
As mentioned before, the first assumption allows us to ignore the strictly semistable locus when we are working with divisors, and the second assumption lets us use properties of the Picard group that we need.
These assumptions are justified for a few reasons.
First, they hold for the Hilbert schemes of points, which are the primary examples of such moduli spaces.
Second, Yoshioka proved that the second assumption holds for $M(\xi)$ where one of the slope components of $\xi$ is an integer and $\xi$ is above the $\delta$ surface \cite{Y:Y96}.
Lastly, there are no examples of $M(\xi)$ where $\xi$ is above the $\delta$ surface for which either assumption is known to fail.
In fact, both assumptions are believed to be true above the $\delta$ surface.
Proving that they hold in this region is the focus of current research.

We also assume that $\xi$ is above the $\delta$ surface and is not a multiple of an exceptional Chern character.

\subsection{The Picard Group of $M(\xi)$}
\label{subsec: the Picard group of M}
We have a good description of the Picard group of $\MMM$ if we assume that the Picard rank is three which we just added as a standing assumption.
Linear and numerical equivalence coincide on $\MMM$, so we have
$$\textnormal{NS}(\MMM) =\textnormal{Pic}(\MMM) \otimes \RR .$$
As mentioned above, we will work with $M^s(\xi)$ when it is convenient.
$M^s(\xi)$ is a coarse moduli space for the stable sheaves.
In contrast, $\MMM$ is not a coarse moduli space, unless $\xi$ is a primitive character, as it identifies $S$-equivalence classes of (strictly semistable) sheaves.

In order to understand the classes of the divisors that will span the effective cone, we have to understand the isomorphism 
$$\xi^{\perp} \isom \textnormal{NS}\left( \MMM\right).$$
We construct this isomorphism by uniquely defining a line bundle on families of sheaves for each element of $\xi^\perp$ as done in \cite{LP:LP97} and \cite{CHW:CHW14} for the case of $\PP^2$.  

Let $\UU/S$ be a flat family of semistable sheaves with Chern character $\xi$ where $S$ is a smooth variety.
Define the two projections $p: S \times \QS \to S$ and $q: S\times \QS \to \QS$.
Then consider the composition of maps
$$\lambda_\UU:K(\QS) \xrightarrow{q^*} K(S\times \QS) \xrightarrow{\cdot [\UU]} K(S\times\QS) \xrightarrow{-p!} K(S) \xrightarrow{\det} \pic(S)$$
where $p! = \sum (-1)^iR^iP_*$.
Colloquially, we are taking a sheaf on $\QS$, pulling it back along the projection to $S \times \QS$, restricting it to our family, pushing it forward along the projection onto $S$, and then taking its determinant to get a line bundle on $S$.

Tensoring our family $\UU$ by $p^*L$, where $L$ is a line bundle on $S$, does not change the given moduli map $f: S \to \MMM$, but does change the $\lambda$ map as follows:
$$\lambda_{\UU \otimes p^* L}(\zeta) = \lambda_{\UU}(\zeta) \otimes L^{\otimes -(\xi, \zeta)}.$$

Now given a class $\zeta \in \xi^\perp$, we want a $\lambda_M$ map which commutes with the moduli map in the sense that we should have
$$\lambda_{\UU}(\zeta) = f^* \lambda_M(\zeta)$$
for all $\UU$.
This equality determines a unique line bundle $\lambda_M(\zeta)$ on $\MMM$ and gives a linear map $\lambda_M:\xi^\perp \to \ns(\MMM)$ which is an isomorphism under our assumptions.

Caveat: We have normalized $\lambda_M$ using $-p!$ as in \cite{CHW:CHW14} rather than $p!$ as in \cite{HuL:HuL10} and \cite{LP:LP97} so that the positive rank Chern characters form the ``primary half space" that we define later this chapter. 


\subsection{A Basis for the Picard Group}
\label{subsec:a basis for the Picard space}
Using this isomorphism, we want to construct a basis for the Picard group.
Following Huybrechts and Lehn \cite{HuL:HuL10} and modifying it for $\QS$, we define three Chern characters $\zeta_{0}$, $\zeta_{a}$, and $\zeta_{b}$.
Bundles with these Chern characters will be a basis for the Picard space.
These Chern characters depend on the character $\xi$ of the moduli space. 
Let $a$ be the Chern character of a line of the first ruling and $b$ be the Chern character of a line of the second ruling on $\QS$.
We define our Chern characters by the equations
$$\zeta_{0} = r(\xi)\OO_{\QS}-\chi(\xi^*,\OO_{\QS})\OO_p,$$
$$\zeta_{a} = r(\xi)a_1-\chi(\xi^*,a)\OO_p\textnormal{, and}$$
$$\zeta_{b} = r(\xi)b_1-\chi(\xi^*,b)\OO_p$$
where $\OO_p$ is the structure sheaf of a point on $\QS$.
(Note the difference from the analogous definition in \cite{HuL:HuL10} by a sign due to our convention for $\lambda_M$.)

Given that $(\xi,\OO_{p}) = r(\xi)$, it should be clear that they are all in $\xi^\perp$.
They can also be shown to be in $\xi^\perp$ by using Riemann-Roch given that the Chern characters $(\ch_0, \ch_1, \ch_2)$ of $\zeta_{0},\zeta_{a}, \zeta_{b}$ are
$$\zeta_{0} = \left( r(\xi), (0, 0), -\chi(\xi) \right),$$
$$\zeta_{a} = \left( 0, \left(r(\xi),0 \right), - r(\xi)-c_1(\xi)\cdot(0,1)\right) ,  \textnormal { and}$$
$$\zeta_{b} = \left( 0, \left(0,r(\xi)\right), - r(\xi)-c_1(\xi)\cdot(1,0) \right).$$

Using the map from the previous section, define $\LL_{0}$, $\LL_{a}$, $\LL_{b}$ by the formulas $\LL_{0} = \lambda_M(\zeta_{0})$, $\LL_{a} = \lambda_M(\zeta_{a})$, and $\LL_{b} = \lambda_M(\zeta_{b})$. 
They are a basis for the Picard space.
%
We know that for $n>>0$, $\LL_{0} \otimes \LL_{a}^n$ and $\LL_{0} \otimes \LL_{b}^n$ are ample.

Now, $\zeta_a$ and $\zeta_b$ span the plane of rank zero sheaves in $\xi^\perp$.
Define the \textit{primary half space} of $\xi^\perp$ to be the open half space of positive rank Chern characters and the \textit{secondary half space} to be the closed half space of negative rank and rank zero Chern characters in $\xi^\perp$.
Similarly, define the \textit{primary and secondary halves} of $\ns(\MMM)$ as the images of these under the isomorphism $\lambda_M$.
Every extremal ray of the effective and nef cones sits in one of our halves.
Call an extremal ray of the effective or nef cone \textit{primary} or \textit{secondary} according to which half it lies in.

We know that the ample cone of $\MMM$ lies in the primary half space because $\LL_0$ lies in that half space.

\subsection{Brill-Noether Divisors}
\label{subsec: BN divisors}
Now that we have constructed a basis for the Picard space, we discuss the divisors that we will construct to span the effective cone.
These divisors will be examples of \textit{Brill-Noether divisors}.
A Brill-Noether locus in general is the place where the rank of some cohomology group jumps.

\begin{proposition}\label{prop:brillnoether}
Suppose $V \in M(\zeta)$ is a stable vector bundle and is cohomologically orthogonal to the general $U \in \MMM$. 
Put the natural determinantal scheme structure on the locus $$D_V = \overline{\{ U \in \MS : h^1(U\otimes V) \neq 0 \}}$$.

(1) $D_V$ is an effective divisor.

(2) If $\mu_{1,1}(U\otimes V)>-4$, then $\OO_{\MMM}(D_V) \isom \lambda_M(\zeta)$.

(3) If $\mu_{1,1}(U\otimes V)<0$, then $\OO_{\MMM}(D_V) \isom \lambda_M(\zeta)^* \isom \lambda_M(-\zeta)$.
\end{proposition}

\begin{proof}
After replacing slope with $(1,1)$-slope and $-3$ with $-4$ (the first is the slope of $K_{\PP^2}$
while the second is the $(1,1)$-slope of $K_{\QS}$), the proof is identical to that of Prop 5.4 in \cite{CHW:CHW14}.
\end{proof}

Given a Brill-Noether divisor, a natural question to ask is whether it lies in the primary or the secondary half of $\ns(\MMM)$.
The answer is immediate from the computation of the class of the Brill-Noether divisors.

\begin{cor}
\label{cor: brill noether half space}
Keep the notation and hypotheses from Prop. \ref{prop:brillnoether}.\newline
(1) If $\mu_{1,1}(U\otimes V) > -4$, then [$D_V$] lies in the primary half of $\textnormal{NS}(M(\xi))$. \newline
(2) If $\mu_{1,1}(U\otimes V) < 0$, then [$D_V$] lies in the secondary half of $\textnormal{NS}(M(\xi))$.
\end{cor}

Note that since $\mu_{1,1}(U\otimes V)$ cannot be between $-4$ and $0$ there is no contradiction in these results.  

\subsection{The Interpolation Problem}\label{subsec: the interpolation problem}

As we previously mentioned, in order to construct a Brill-Noether divisor, it is necessary to find a sheaf $V$ for which $h^1(U\otimes V)= 0$ for the general $U \in \MMM$ and the easiest way to do so is to find a sheaf that is cohomologically orthogonal to $U$.
Cohomological orthogonality implies that $\chi(U \otimes V) = 0$.
The vanishing of that Euler characteristic is a strictly weaker condition.
For example, it might be the case that $h^0 = h^1 =1$ and $h^2=0$.
We would like an added condition which would make them equivalent. 
A bundle $V$ is \textit{non-special with respect to} $U$ if $\chi(U \otimes V)$ determines the cohomology groups (i.e. they are lowest ranks allowed by the Euler characteristic).
We can rephrase cohomological orthogonality as $V$ having $\chi(U \otimes V) = 0$ and being non-special with respect to $U$.
In general, there will be many such sheaves, but those that interest us will be those of ``minimum slope."
Finding a sheaf like this is a form of the interpolation problem.

\begin{intprob*}
Given invariants $\xi$ of a vector bundle and a polarization $H$ of $\QS$, find a stable vector bundle $V$ with minimum $\mu_H$ that is cohomologically orthogonal to the general element of $M(\xi)$.
\end{intprob*}
Note, if we restrict our interest to line bundles on $\QS$ and $M(\xi) = \left(\QS\right)^{[m]}$, this is the classical interpolation for points lying on $\QS$: find the ``lowest" bidegree $(a,b)$ such that $m$ points lie on a curve of type $(a,b)$.

By construction of the Brill-Noether divisors, any solution of the interpolation problem will give an effective Brill-Noether divisor.
Our goal is to give a method to construct Brill-Noether divisors which are sufficient to span the effective cone in many examples.
We show that they span in examples by providing an alternate construction of the divisors.
This description provides dual moving curves which prove that they are extremal divisors in the effective cone.

The alternate construction starts by using the cohomological vanishing guaranteed by these Brill-Noether divisors to resolve the general objects of $\MMM$.
In general, we then use these resolutions to construct maps from $\MMM$ to Picard rank one varieties that have positive dimensional fibers.
The Brill-Noether divisors will match up with the pull backs of an ample divisor on these Picard rank one varieties.
We will prove the extremality of each divisor using the dual moving curves that cover the fibers of the morphism. 

\subsection{Kronecker Moduli Spaces}
\label{subsec:kronecker modules}
The simpler varieties that we will map to are moduli spaces of \textit{Kronecker modules}.

The quivers that we will be interested in have two points and $N$ arrows in one direction between the points.
If there are multiple, say $N$, arrows with the same head and tail, we will only draw one arrow in our quivers but label that arrow with $N$.
$$
\begin{tikzpicture}[scale=3]
\node at (-1/2,0) {$p_0$};
\node at (1/2,0) {$p_1$};
\node at (0,0.1) {$N$};
\draw [->] (-.4,0) -- (.4,0) node[above] {};
\end{tikzpicture}
$$
\begin{deff} This type of quiver, with two vertices and arrows in only direction between them, is a \textit{Kronecker quiver}. \end{deff}
\begin{deff} A \textit{Kronecker $V$-module} is a representation of this quiver and is equivalent to a linear map $A \otimes V \to B$ where $V$ is a vector space of dimension $N$ and $A$ and $B$ are arbitrary vector spaces. \end{deff}
\begin{deff} The moduli space of semistable Kronecker $V$-modules with dimension vector $r = (a,b)$ is $Kr_N(a,b)$. \end{deff}
The \textit{expected dimension} of this space is
$$(\textnormal{edim}) \textnormal{ } Kr_N(a,b) = 1- \chi(r,r) = Nab-a^2-b^2+1.$$
For Kronecker moduli spaces, we also know that they are nonempty and of the expected dimension if their expected dimension is nonnegative \cite{Re:Re08}.
Another fact about Kronecker moduli spaces is that they are Picard rank one \cite{D:D87}.
We use this fact later when we create effective divisors on our moduli spaces $\MMM$ by constructing fibrations from them to these Kronecker moduli spaces and pulling back a generator of the ample cone of $Kr_N(a,b)$.
In order to do this, we need to get Kronecker modules from maps between \textit{exceptional bundles}.

\subsubsection{Kronecker Modules from Complexes}
Given a complex $W:A^a\to B^b$ where $\{A,B\}$ is an exceptional pair, we get a Kronecker $\Hom(A,B)$-module $R$ with dimension vector $r = (a,b)$.
The properties of exceptional bundles tell us that homomorphisms of the Kronecker module are exactly the homomorphisms of $W$ and that $\chi(r,r') = \chi(W,W')$ where $R'$ is the Kronecker module corresponding to a complex $W':A^{a'}\to B^{b'}$.
We will get these complexes between exceptional bundles from resolutions of the general objects of our moduli spaces $\MMM$.


\section{Corresponding Exceptional Pairs}\label{sec:excpair}
In this section, we use exceptional pairs to identify the Brill-Noether divisors that we expect to span the effective cone by identifying possible solutions to the interpolation problem.

Let $U \in \MMM$ be a general element.
Recall that to solve the interpolation problem we wanted to find bundles that were cohomologically orthogonal to $U$ and that being cohomologicaly orthogonal is equivalent to $V$ having $\chi(U \otimes V) = 0$ and $V$ being cohomologically non-special with respect to $\xi$.
Now, we are able find all cohomologically orthogonal bundles by imposing each of these two conditions separately.


\subsection{Bundles with Vanishing Euler Characteristic}
\label{subsec: bundles with vanishing Euler characteristic}
First, we find bundles $V$ with $\chi(U \otimes V) = 0$.
As we are trying to compute the effective cone, scaling Chern characters is relatively unimportant.
In particular, we can scale a Chern character so that $\ch_0 =1$ (unless it was 0 to start). 
\begin{definition}\label{def: orthogonal surface}
When a Chern character $\xi$ has positive rank, the \textit{orthogonal surface to $\xi$} is
$$Q_\xi = \{ (\mu,\Delta) : (1,\mu,\Delta) \textnormal{ lies in } \xi^\perp \} \subset \RR^3.$$
\end{definition}
Using the orthogonal surface rather than the full $\xi^\perp$ has the advantage of working in the three dimensional $\left( \mu, \Delta\right)$-space instead of in the full four dimensional $K(\QS)$.
We define the \textit{reference surface}, $Q_{\xi_0}$, to be the orthogonal surface to $\xi_0 = \ch(\OO) = (1,(0,0),0)$.

Using Hirzebruch-Riemann-Roch to compute the equation of the orthogonal surface to $\xi$ gives the formula 
$$Q_\xi : P(\mu(\xi) +\mu) -\Delta(\xi) = \Delta$$
where $P(x,y) = (x+1)(y+1)$. 
This equation defines a saddle surface that is a shift of the reference surface.
It has unique saddle point at the point $\left(-1-x_0,-1-y_0\right)$ where $(x_0,y_0) = \mu(\xi)$.
Consequently, any two such surfaces intersect in a parabola that lies over a line in the slope plane.

Using this language, the condition that $\chi(U \otimes V) = 0$ can be rephrased as saying that $V$ must lie on $Q_\xi$.
In other words, we can restrict our search for solutions to the interpolation problem to bundles that lie on the orthogonal surface (to $\xi$).

As a side note, we can now give an alternative description of each fractal part of the $\delta$ surface using orthogonal surfaces. 
Given an exceptional bundle $E_\alpha$, $\delta_E$ can equivalently be written as 
\begin{displaymath}
   \delta_E(\mu) = \left\{
     \begin{array}{lr}
       Q_{E^*}(\mu) & \textnormal{ if } \mu_{1,1}(E)-4 <\mu_{1,1}(\mu) \textnormal{ and } \gbar(\mu)<\gbar(E) \\
       Q_{E^*(K)}(\mu) & \textnormal{ if } \mu_{1,1}(\mu) <\mu_{1,1}(E) +4  \textnormal{ and } \gbar(E)<\gbar(\mu) 
     \end{array}
   \right.
\end{displaymath}
Thus, every saddle subsurface of the surface $\delta(\mu) =\Delta$ can be seen to be a portion of some surface $Q_{E_\alpha}$.


\subsection{Cohomologically Non-special Bundles}
\label{subsec: cohomologically non-special bundles}
Now that we have found the bundles with $\chi(U \otimes V) = 0$, we would like to impose the second condition of cohomological orthogonality, being \textit{non-special}.
In some cases, we can find numerical conditions defining being non-specialty as well.
Those conditions will be in terms of certain exceptional bundles that we have to pick out.
Picking them out begins with studying $Q_\xi$ again.

Note that $Q_\xi$ intersects the plane $\Delta =\frac{1}{2}$.
The exceptional bundles that we want to pick out are those exceptional bundles that control this intersubsection.

\begin{definition}\label{def:controllingexceptional}
A \textit{controlling exceptional bundle} of $\xi$ is an exceptional bundle, $E_{\alpha}$, for which there exists a slope $\nu$ for which $\delta(\nu) = \delta_{E_\alpha}(\nu)$ and $Q_\xi(\nu) = \frac{1}{2}$.
\end{definition}

As promised above, each controlling exceptional bundle will provide a necessary condition for a stable Chern character to be non-special with respect to the general element of $\MMM$.
The condition that a controlling exceptional bundles imposes on a Chern character is that the Chern character must be on or above a surface corresponding to that exceptional bundle.
\begin{definition}\label{def:correspondingsurfaces}
The \textit{corresponding surface} to an exceptional bundle $E_\alpha$ for $\xi$ is defined as
\begin{displaymath}
   Q_{\alpha,\xi}(\nu) = \left\{
     \begin{array}{lr}
       Q_{E_\alpha^*}(\nu) & : \textnormal{if } \chi\left(E_\alpha^*,U\right) >0\textnormal{ }\\
       Q_{E_\alpha^*(K)}(\nu) & : \textnormal{if } \chi\left(E_\alpha^*,U \right) <0.
     \end{array}
   \right.
\end{displaymath}
\end{definition}
Then in some, if not all, cases, the Chern characters $\nu$ such that $Q_{\alpha,\xi}(\nu)>0$ for all controlling bundles $E_\alpha$ of $\xi$ are precisely the non-special Chern characters with respect to $\xi$.


\subsection{Potential Extremal Rays}
\label{subsec: potential extremal rays}
The solutions to interpolation are the intersubsection of the orthogonal surface and the maximum of the corresponding surfaces.
Each part of this intersubsection is where the orthogonal surface intersects a corresponding surface, i.e. where $Q_\xi = Q_{\alpha,\xi}$ for a controlling exceptional $E_\alpha$ of $\xi$. 
The corners of the intersubsection are where the orthogonal surface intersects two different corresponding surfaces, i.e. where $Q_\xi = Q_{\alpha,\xi} = Q_{\beta,\xi}$ for controlling exceptionals $E_\alpha$ and $E_\beta$ of $\xi$.

As $\alpha$, $\beta$, and $\xi$ are three linearly independent rational Chern characters, $Q_\xi \cap Q_{\alpha,\xi} \cap Q_{\beta,\xi}$ is a single point that corresponds to the intersubsection of the 3-planes $\alpha^\perp \cap \beta^\perp \cap \gamma^\perp$, $(\alpha+K)^\perp \cap \beta^\perp \cap \gamma^\perp$, $\alpha^\perp \cap (\beta+K)^\perp \cap \gamma^\perp$, or $(\alpha+K)^\perp \cap (\beta+K)^\perp \cap \gamma^\perp$. 
We determine which intersubsection it is by which cases of $Q_{\alpha,\xi}$ and $Q_{\beta,\xi}$ we are using.
We want to find the corners made in this way in order to get effective divisors.
To find those divisors, we first define all of the possible triple intersubsections that might work. 

\begin{definition}\label{def:corner}
The \textit{corresponding orthogonal point} of a pair of controlling exceptional bundles $E_\alpha$ and $E_\beta$ is one of the following\\
(1) the unique point $\left(\mu^+,\Delta^+ \right) \in Q_\xi \cap Q_{\alpha,\xi} \cap Q_{\beta,\xi}$,\\
(2) $\beta$ if $\beta \in Q_\xi \cap Q_{\alpha,\xi}$, or \\
(3) $\alpha$ if $\alpha \in Q_\xi \cap Q_{\beta,\xi}$.\\
\end{definition}
(2) and (3) can occur simultaneously, but you can treat them individually so we will say ``the orthogonal point".

Some of the corresponding orthogonal points will not actually be what we want, as they can both fail to satisfy interpolation (by being below one of the other corresponding surfaces) or by not being extremal (as they have a slope that is not linearly independent of other solutions).
We want to pick out the only the solutions that satisfy interpolation and are extremal among those solutions.

\begin{definition}\label{def:controllingpair}
A \textit{controlling pair} of $\xi$ is an exceptional pair of controlling exceptional bundles $E_\alpha$ and $E_\beta$ of $\xi$ with a corresponding orthogonal point $\left( \mu^+, \Delta^+ \right)$ that is stable. 
\end{definition}

We now want to turn these points back into Chern characters for convenience.
\begin{definition}\label{def:orthogonalcherncharacter}
A \textit{potential (primary) orthogonal Chern character} $\xi^+$ to $\xi$ is defined by any character $\xi^+ = \left(r^+,\mu^+,\Delta^+ \right)$ where $r^+$  is sufficiently large and divisible and $(\mu^+,\Delta^+)$ is the corresponding orthogonal point of a controlling pair of $\xi$.
\end{definition}

We call them \textit{potential} because we will see in the next section that we need a few additional conditions to make sure that they actually span an extremal ray of the effective cone.
We will give an approach to showing that some of the potential primary orthogonal Chern characters span the solutions of the interpolation problem and the effective cone of $\MMM$ is spanned by the Brill-Noether divisors $D_V$ for many examples of $\MMM$ where $V$ are bundles whose Chern character is an orthogonal Chern character $\xi^+$.

The behavior of these extremal rays depends on the sign of $\chi(E_{\alpha}^*,U)$ and $\chi(E_{\beta}^*,U)$.
Keeping the identification of $\textnormal{NS}\left( M(\xi) \right) \isom \xi^\perp$ in mind, recall that the primary half of the space corresponds to characters of positive rank.
\newline
(1) If $\chi(\xi_{-\alpha},\xi) > 0$ and $\chi(\xi_{-\beta},\xi) > 0$, the ray is spanned by a positive rank Chern character in $\xi^\perp \cap \xi_{-\alpha}^{\perp} \cap \xi_{-\beta}^{\perp}$. 
\newline
(2) If $\chi(\xi_{-\alpha},\xi) < 0$ and $\chi(\xi_{-\beta},\xi) > 0$, the ray is spanned by a positive rank Chern character in $\xi^\perp \cap \left(\xi_{-\alpha+K}\right)^{\perp} \cap \xi_{-\beta}^{\perp}$. 
\newline
(3) If $\chi(\xi_{-\alpha},\xi) < 0$ and $\chi(\xi_{-\beta},\xi) < 0$, the ray is spanned by a positive rank Chern character in $\xi^\perp \cap \left(\xi_{-\alpha+K}\right)^{\perp} \cap \left(\xi_{-\beta+K}\right)^{\perp}$. 
\newline
(4) If $\chi(\xi_{-\alpha},\xi) = 0$ or $\chi(\xi_{-\beta},\xi) = 0$, the ray is spanned by $\alpha$ or $\beta$, respectively.


\section{The Beilinson Spectral Sequence}
\label{sec: Beilinson Spec seq}
In this section, we find a resolution of the general object $U \in \MMM$ for each orthogonal Chern character via the generalized Beilinson spectral sequence.
In the next section, we use these resolutions to construct fibrations $\MMM \dashrightarrow Kr_V(m,n)$.
In the second last section, these fibrations will give us effective divisors on $\MMM$.

An orthogonal Chern character already has exceptional pairs associated to it.
In order to use the spectral sequence, we have to complete each of those exceptional pairs to a coil.

The coil we use to resolve $U$ depends on the behavior of the extremal ray that we exhibited in the last section.
\newline
(1) If $\chi\left( E_{-\alpha},U \right)\geq 0$ and $\chi\left( E_{-\beta},U \right)\geq 0$, we will decompose $U$ according to the coil 

$\left(F_0^*,F_{-1}^*,E_{-\beta},E_{-\alpha}\right)$ where $(F_0,F_1)$ is a minimally ranked right completion pair of $(E_\alpha,E_\beta)$.
\newline
(2) If $\chi\left( E_{-\alpha},U \right)\leq 0$ and $\chi\left( E_{-\beta},U \right)\geq 0$, we will decompose $U$ according to the coil 

$\left(E_{-\alpha}(K),F_0^*,F_{-1}^*,E_{-\beta}\right)$ where $(F_0,F_1)$ is a minimally ranked right completion pair of $(E_\alpha,E_\beta)$.
\newline
(3) If $\chi\left( E_{-\alpha},U \right)\leq 0$ and $\chi\left( E_{-\beta},U \right)\leq 0$, we will decompose $U$ according to the coil 

$\left(E_{-\beta}(K),E_{-\alpha}(K),F_0^*,F_{-1}^* \right)$ where $(F_0,F_1)$ is a minimally ranked right completion pair of $(E_\alpha,E_\beta)$.


The spectral sequence will only give a resolution under some assumptions on the controlling pairs.
These are the additional conditions needed to make a potential orthogonal Chern character span an extremal ray of the effective cone.
We call those controlling pairs that satisfy the needed conditions \textit{extremal pairs}; there will be a different definition for if a controlling pair is extremal based on the signs of some Euler characteristics so we defer the definition to the following three subsections.

Extremal pairs pick out the exact Chern characters that correspond to extremal effective divisors using our approach as promised in the previous section.
\begin{definition}
\label{def: primary orthogonal chern character}
A \textit{primary orthogonal Chern character} is the potential primary orthogonal Chern character associated to an extremal pair
\end{definition}

It is a current area of research to show the following conjecture for extremal pairs (including those as defined analogously in the next three subsections).
\begin{conj}
Every $\xi$ above the Rudakov $\delta$-surface has an extremal pair and every controlling exceptional bundle of it that is in an extremal pair is in two extremal pairs.
\end{conj}
Proving the conjecture would show that the process laid out in this paper computes the effective cone of $\MMM$ for all $\xi$ above the $\delta$ surface.

\subsection{The ``Mixed Type" Spectral Sequence.}
\label{subsec:spec seq neg and pos}
Assume $\chi\left( E_{-\alpha},U \right) \leq 0$ and $\chi\left( E_{-\beta},U \right) \geq 0$.
Let the right mutated coil of $\left(E_{-\alpha}(K),F_0^*,F_{-1}^*,E_{-\beta}\right)$ be $\left(E_{-\beta},E_{-1}^*,E_{-2}^*,E_{-\alpha}\right)$.
Let $\Delta_i$ be as in the spectral sequence with input these two coils.
\begin{definition}\label{def:extremal pair}
A (mixed type) controlling pair of $\xi$, $(E_\alpha,E_\beta)$, with corresponding orthogonal slope and discriminant $\left( \mu^+, \Delta^+ \right)$ is called \textit{extremal} if it satisfies the following conditions:\newline
(1) They are within a unit in both slope coordinates. \newline
(2) $\mu_{1,1}(E_\alpha)$, $\mu_{1,1}(E_{-2})$, $\mu_{1,1}(E_{-1})$, and $\mu_{1,1}(E_\beta)$ are all greater than $\mu_{1,1}(U)-4$.\newline
(3) ($\Delta_2 = 1$ and $\chi\left( E_{-2}^*,U\right)\geq 0$) or ($\Delta_2 = 0$ and $\chi\left( E_{-2}^*,U\right)\leq 0$).\newline
(4) ($\Delta_1 = 1$ and $\chi\left( E_{-1}^*,U\right)\geq 0$) or ($\Delta_1 = 0$ and $\chi\left( E_{-1}^*,U\right)\geq 0$).\newline
(5)$\sheafhom(E_{-\alpha}(K),F_{-1}^*) \textnormal{, } \sheafhom(E_{-\alpha}(K),E_{-\beta}) \textnormal{, } \sheafhom(F_0^*,F_{-1}^*) \textnormal{, and } \sheafhom(F_0^*,E_{-\beta})$
are all globally generated.\newline
(6) Any bundle sitting in a triangle $  \left(F_{-1}^*\right)^{m_{1}} \osum E_{-\beta}^{m_{0}} \to U \to \left(E_{-\alpha}(K)^{m_{3}} \osum \left( F_{0}^*\right)^{m_{2}}\right)[1]$ is prioritary.
\end{definition}

Given an extremal pair, we can resolve the general object of $\MMM$.

\begin{thm}\label{thm:resolution-neg pos}
The general $U \in M(\xi)$ admits a resolution of the following form 
$$0  \to E_{-\alpha}(K)^{m_{3}} \osum \left( F_{0}^*\right)^{m_{2}} \fto{\phi} \left(F_{-1}^*\right)^{m_{1}} \osum E_{-\beta}^{m_{0}} \to U \to 0.$$
\end{thm}

\begin{proof}
Consider a bundle $U$ defined by the sequence 
$$0  \to E_{-\alpha}(K)^{m_{3}} \osum \left( F_{0}^*\right)^{m_{2}} \fto{\phi} \left(F_{-1}^*\right)^{m_{1}} \osum E_{-\beta}^{m_{0}} \to U \to 0$$
where the map $\phi \in \Hom\left(E_{-\alpha}(K)^{m_{3}} \osum \left( F_{0}^*\right)^{m_{2}}, \left(F_{-1}^*\right)^{m_{1}} \osum E_{-\beta}^{m_{0}}\right)$ is general.

The proof proceeds in 4 steps: calculate that $\ch(U) = \xi$, show $\phi$ is injective, confirm the expected vanishings in the spectral sequence, and prove that $U$ is stable.

\textit{Step 1:} Calculate $\ch(U) = \xi$. 
We do not know if $\phi$ is injective yet, but we can compute the Chern character of the mapping cone of $\phi$ in the derived category. 
Assuming $\phi$ is injective, this computes the Chern character of $U$.

This computation follows from the generalized Beilinson spectral sequence's convergence to $U$. 
Specifically, we have a spectral sequence with $E_1$-page 
\begin{align*}
E_{-\alpha}(K) \otimes \CC^{m_{32}} & \to F_0^* \otimes \CC^{m_{22}} \to F_{-1}^* \otimes \CC^{m_{12}} \to 0 \\
E_{-\alpha}(K) \otimes \CC^{m_{31}} & \to F_0^* \otimes \CC^{m_{21}} \to F_{-1}^* \otimes \CC^{m_{11}} \to  E_{-\beta} \otimes \CC^{m_{01}} \\
0 & \to F_0^* \otimes \CC^{m_{20}} \to F_{-1}^* \otimes \CC^{m_{10}} \to E_{-\beta} \otimes \CC^{m_{00}} \\
\end{align*}
that converges in degree 0 to the sheaf $U$ and to 0 in all other degrees.
We omitted all nonzero rows and columns.
In particular, row 3 is zero by the vanishing of $h^2$ for all the sheaves (follows from Def. \ref{def:extremal pair}). 
Also note that either the top or bottom element of each of the middle two rows vanishes (depending on the value of $\Delta_1$ and $\Delta_2$). 
An easy computation shows that 
\begin{align*}\ch(U) = &-(m_{32}-m_{31})\ch(E_{-\alpha}(K))+(m_{22}-m_{21}+m_{20})\ch(F_0^*)\\ &-(m_{12}-m_{11}+m_{10})\ch(F_{-1}^*)+(m_{00}-m_{01})\ch(E_{-\beta}).\end{align*}
In our situation, we see that this gives
$$\xi = -m_{-3}\ch(E_{-\alpha}(K))+m_{-2}\ch(F_0^*)-m_{-1}\ch(F_{-1}^*)+m_0\ch(E_{-\beta})$$
where the $m_i$ are defined in the obvious way.

\textit{Step 2:} Show $\phi$ is injective. 
The sheaves 
$$\sheafhom(E_{-\alpha}(K),F_{-1}^*) \textnormal{, } \sheafhom(E_{-\alpha}(K),E_{-\beta}) \textnormal{, } \sheafhom(F_0^*,F_{-1}^*) \textnormal{, and } \sheafhom(F_0^*,E_{-\beta})$$
are all globally generated by the controlling pair being an extremal pair.
Those bundles being globally generated immediately implies that
$$\sheafhom(E_{-\alpha}(K)^{m_{3}}\osum \left(F_0^*\right)^{m_{2}}, \left(F_{-1}^*\right)^{m_{1}} \osum E_{-\beta}^{m_0})$$
is globally generated as well.
Using a Bertini-type theorem [Prop 2.6, \cite{Hu:Hu16}] and the fact that the virtually computed rank of $U$ is positive, we see that $\phi$ is injective.
\newline
\textit{Step 3:} Verify $U$'s spectral sequence has the correct vanishings.
We know that $\chi(E_{-\alpha},U) \leq 0$ and $\chi(E_{-\beta},U) \geq 0$. 
We also know that ($\chi(E_{-2}^*,U) \leq 0$ and $\Delta_2=0$) or ($\chi(E_{-2}^*,U) \geq 0$ and $\Delta_2=1$).
Analogously, we know that ($\chi(E_{-1}^*,U) \leq 0$ and $\Delta_1=0$) or ($\chi(E_{-1}^*,U) \geq 0$ and $\Delta_1=1$).

Since we know that all of the groups other than possibly $\ext^1$ and $\hom$ vanish, it is enough to check that 
$$\hom(E_{-\alpha},U) = \ext^{\Delta_1}(E_{-1}^*,U) = \ext^{\Delta_2}(E_{-2}^*,U) = \ext^1(E_{-\beta},U) = 0.$$
These vanishings for the specific $U$ we have resolved will follow from the orthogonality properties of exceptional bundles and the relevant long exact sequences.
Once they vanish for a specific $U$, they will vanish for a general $U$ as needed.
We show these four vanishings in order.

First, $\hom(E_{-\alpha}, F_{-1}^*)$, $\hom(E_{-\alpha}, E_{-\beta})$, and $\ext^1(E_{-\alpha}, F_0^*)$ all vanish since $\left( F_0^*,F_{-1}^*, E_{-\beta},E_{-\alpha}\right)$ is a coil and  $\ext^1(E_{-\alpha}, E_{-\alpha}(K)) = \ext^1(E_{-\alpha}(K),E_{-\alpha}(K))$ which vanishes since exceptional bundles are rigid.
This gives $\hom(E_{-\alpha},U) = 0$ as desired.

For the next vanishing, we have two cases: $\Delta_1 = 0$ and $\Delta_1=1$.
We will assume that $\Delta_1=0$, and the other case is similar.
Assuming $\Delta_1=0$, we need to show that $\hom(E_{-1}^*, U)=0$.
Next, $\ext^1(E_{-1}^*, E_{-\alpha}(K))$, $\hom(E_{-1}^*, E_{-\beta})$, and $\ext^1(E_{-1}^*, F_0^*)$ all vanish since $\left( E_{-\alpha}(K),F_0^*, E_{-\beta},E_{-1}^*\right)$ is a coil.
Then the only remaining vanishing we need to prove is $\hom(E_{-1}^*,F_{-1}^*)$.
By assumption on $\Delta_1$, we can write $0 \to F_{-1}^* \to E_{-\beta}^a \to E_{-1}^* \to 0$ where $a= \chi(E_{-\beta},E_{-1})$.
From this resolution, we see $\hom(E_{-1}^*,F_{-1}^*) \hookrightarrow \hom(E_{-1}^*,E_{-\beta}) = 0$, and the vanishing follows.

For the next vanishing, we have two cases: $\Delta_1 = 0$ and $\Delta_1=1$.
We will assume that $\Delta_1=0$, and the other case is similar.
Assuming $\Delta_2=0$, we need to show that $\hom(E_{-2}^*, U)=0$.
Then, $\ext^1(E_{-2}^*, E_{-\alpha}(K))$, $\hom(E_{-2}^*, E_{-\beta})$, and $\hom(E_{-2}^*, F_{-1}^*)$ all vanish since $\left( E_{-\alpha}(K),F_{-1}^*, E_{-\beta},E_{-2}^*\right)$ is a coil.
Then the only remaining vanishing we need to prove is $\ext^1(E_{-2}^*,F_0^*)$.
By assumption on $\Delta_2$, we can write $0 \to F_0^* \to E_{-\beta}^a \to L_{E_{-1}^*}E_{-2}^* \to 0$ where $a = \chi(E_{-\beta},L_{E_{-1}^*}E_{-2}^*)$.
From this resolution we see that $\ext^1(E_{-2}^*,F_0^*) \isom \hom(E_{-2}^*,L_{E_{-1}^*}E_{-2}^*)$ since $\hom(E_{-2}^*, E_{-\beta}) = \ext^1(E_{-2}^*, E_{-\beta})= 0$.
Thus, we reduce to showing that $\hom(E_{-2}^*,L_{E_{-1}^*}E_{-2}^*)=0$.
Again, by assumption on $\Delta_2$, we know that we can write $0 \to L_{E_{-1}^*}E_{-2}^* \to (E_{-1}^*)^a \to E_{-2}^* \to 0$ from which we can easily see $\hom(E_{-2}^*,L_{E_{-1}^*}E_{-2}^*) \hookrightarrow \hom(E_{-2}^*,E_{-1}^*) = 0$, and the vanishing follows.

Finally, we will show $\ext^1(E_{-\beta},U) = 0$. 
Then $\ext^1(E_{-\beta}, F_{-1}^*)$, $\ext^2(E_{-\beta}, E_{-\alpha}(K))$, and \\$\ext^2(E_{-\beta}, F_0^*)$ all vanish since $\left(E_{-\alpha}(K), F_0^*,F_{-1}^*, E_{-\beta}\right)$ is a coil while  $\ext^1(E_{-\beta}, E_{-\beta}) = 0 $ since exceptional bundles are rigid.
This gives the vanishing of $\ext^1(E_{-\beta},U)$. 
\newline
\textit{Step 4:} Prove that $U$ is stable.
Let 
$$S \subset \Hom\left(E_{-\alpha}(K)^{m_{3}}\osum \left(F_0^*\right)^{m_{2}},\left(F_{-1}^*\right)^{m_{1}} \osum E_{-\beta}^{m_0}\right)$$
be the open subset of sheaf maps that are injective and have torsion-free cokernels.

By the argument of 5.3 of \cite{CHW:CHW14}, it follows that $S$ is non-empty.

Consider the family $\rfrac{\mathcal{U}}{S}$ of quotients parametrized by $S$.
We need to show that $\mathcal{U}$ is a complete family of prioritary sheaves.
Recall that a prioritary sheaf is a torsion free sheaf $U$ such that
$\Ext^2(U,U(0,-1)) = 0$ or $\Ext^2(U,U(-1,0)) = 0$ or, equivalently in our case, that $\Hom(U,U(-1,-2)) =0$ or $\Hom(U,U(-2,-1)) =0.$
By Def. \ref{def:extremal pair}, the elements of $\mathcal{U}$ are prioritary.
Again by the general argument of 5.3 of \cite{CHW:CHW14}, the family is a complete family.
By Thm. 1 of \cite{Wa:Wa98}, the Artin stack of prioritary sheaves with Chern character $\xi$ is an irreducible stack that contains the stack of semistable sheaves with Chern character $\xi$ as a dense open subset.
It is then clear that $S$ parametrizes the general sheaves in $M(\xi)$.
\end{proof}


\subsection{The ``Negative Type" Spectral Sequence.}
\label{subsec:spec seq neg and neg}

Assume $\chi\left( E_{-\alpha},U \right)\leq 0$ and $\chi\left( E_{-\beta},U \right) \leq 0$.
Let the right mutated coil of $\left(E_{-\beta}(K), E_{-\alpha}(K),F_{0}^*,F_{-1}^*\right)$ be $\left(F_{-1}^*,F_{-2}^*,E_{-\gamma},E_{-\beta}\right)$.
Let $\Delta_i$ be as in the spectral sequence with input these two coils.

\begin{definition}\label{def:neg extremal pair}
A (negative type) controlling pair of $\xi$, $(E_\alpha,E_\beta)$, with corresponding orthogonal slope and discriminant $\left( \mu^+, \Delta^+ \right)$ is called \textit{extremal} if it satisfies the following conditions:\newline
(1) They are within a unit in both slope coordinates. \newline
(2) $\mu_{1,1}(E_\alpha)$, $\mu_{1,1}(F_{-1})$, $\mu_{1,1}(F_{-2})$, and $\mu_{1,1}(E_\beta)$ are all greater than $\mu_{1,1}(U)-4$.\newline
(3) ($\Delta_2 = 0$ and $\chi\left( E_{-\gamma},U\right)\leq 0$) or ($\Delta_2 = 1$ and $\chi\left( E_{-\gamma},U\right)\geq 0$).\newline
(4) $\Delta_1 = 0$, $\chi\left( F_{-1}^*,U\right)\geq 0$, and $\chi\left( F_{-2}^*,U\right)\geq 0$.\newline
(5) $\sheafhom(E_{-\beta}(K),F_{-1}^*) \textnormal{, } \sheafhom(E_{-\alpha}(K),F_{-1}^*) \textnormal{, and }  \sheafhom(F_{0}^*,F_{-1}^*)$ are all globally generated.\newline
(6) Any bundle sitting in a triangle $ F_{-1}^{m_{0}} \to U \to \left( \left(F_{0}^*\right)^{m_{1}} \osum E_{-\alpha}(K)^{m_{2}} \osum E_{-\beta}(K)^{m_{3}}\right)[1]$ is prioritary.
\end{definition}

Given an extremal pair, we can resolve the general object of $\MMM$.

\begin{thm}\label{thm:resolution-neg neg}
The general $U \in M(\xi)$ admits a resolution of the following forms 
$$0  \to E_{-\beta}(K)^{m_{3}} \osum E_{-\alpha}(K)^{m_{2}} \osum \left( F_{0}^*\right)^{m_{1}} \to \left(F_{-1}^*\right)^{m_{0}}  \to U \to 0.$$
\end{thm}
As the proof is similar to the proof of the previous theorem, we omit its proof.

\subsection{The ``Positive Type" Spectral Sequence.}
\label{subsec:spec seq pos and pos}

Assume $\chi\left( E_{-\alpha},U \right)\geq 0$ and $\chi\left( E_{-\beta},U \right)\geq 0$.
Let the right mutated coil of $\left(F_{0}^*,F_{-1}^*,E_{-\beta}, E_{-\alpha}\right)$ be $\left(E_{-\alpha},E_{-\gamma},F_{1}^*(-K),F_{0}^*(-K)\right)$.
Let $\Delta_i$ be as in the spectral sequence with input these two coils.

\begin{definition}\label{def:pos extremal pair}
A (positive type) controlling pair of $\xi$, $(E_\alpha,E_\beta)$, with corresponding orthogonal slope and discriminant $\left( \mu^+, \Delta^+ \right)$ is called \textit{extremal} if it satisfies the following conditions:\newline
(1) They are within a unit in both slope coordinates. \newline
(2) $\mu_{1,1}(E_\alpha)$, $\mu_{1,1}(F_{1}(K))$, $\mu_{1,1}(F_{0}(K))$, and $\mu_{1,1}(E_\beta)$ are all greater than $\mu_{1,1}(U)-4$.\newline
(3) ($\Delta_1 = 1$ and $\chi\left(E_{-\gamma},U\right)\geq 0$) or ($\Delta_1 = 0$ and $\chi\left(E_{-\gamma},U\right)\leq 0$). \newline
(4) $\chi\left(F_{-1}^*(-K),U\right)\leq 0$, $\chi\left(F_{0}^*(-K),U\right)\leq 0$, and $\Delta_2=1$.\newline
(5) $\sheafhom(F_{-0}^*,F_{-1}^*) \textnormal{, } \sheafhom(F_{-0}^*,E_{-\beta}) \textnormal{, and }  \sheafhom(F_{0}^*,E_{-\alpha})$ are all globally generated.\newline
(6) Any bundle sitting in a triangle $ F_{-1}^{m_{2}} \osum E_{-\beta}^{m_{1}} \osum E_{-\alpha}^{m_{0}}\to U \to \left(F_{0}^*\right)^{m_{3}}[1] $ is prioritary.\\
\end{definition}

Given an extremal pair, we can resolve the general object of $\MMM$.


\begin{thm}\label{thm:resolution-pos pos}
The general $U \in M(\xi)$ admits a resolution of the following forms 
$$0  \to  \left( F_{0}^*\right)^{m_{3}} \to \left(F_{-1}^*\right)^{m_{2}} \osum E_{-\beta}^{m_{1}} \osum E_{-\alpha}^{m_{0}}  \to U \to 0.$$
\end{thm}
As the proof is similar to the proof of the second previous theorem, we omit its proof.


\section{The Kronecker Fibration}
\label{sec: The Kron. Fib.}
In this subsection, we use the resolutions constructed in the last section to construct a map to a moduli space of Kronecker modules.
In the next section, these maps produce effective divisors on $\MMM$.

One thing that we need to know in order for this to work is that all the homomorphisms in the derived category are morphisms of complexes.
\begin{lemma}\label{lem:complexes}
Consider a pair of two term complexes
$$W = E^*\otimes \CC^{m_3} \osum F_0^* \otimes \CC^{m_2} \to F_{-1}^* \otimes \CC^{m_1}$$
and
$$W' = E^*\otimes \CC^{m'_3} \osum F_0^* \otimes \CC^{m'_2} \to F_{-1}^* \otimes \CC^{m'_1}$$
each sitting in degrees 0 and -1.
Every homomorphism $W \to W'$ in the derived category $D^b(\QS)$ is realized by a homomorphism of the complexes, so
$$\Hom_{D^b(\QS)}(W,W') = \Hom_{Kom(\QS)}(W,W').$$ 

Similarly, consider a pair of two term complexes
$$W = F_0^* \otimes \CC^{m_2} \to F_{-1}^* \otimes \CC^{m_1} \osum E^*\otimes \CC^{m_0} $$
and
$$W' = F_0^* \otimes \CC^{m'_2} \to F_{-1}^* \otimes \CC^{m'_1} \osum E^*\otimes \CC^{m'_0} $$
each sitting in degrees 0 and -1.
Every homomorphism $W \to W'$ in the derived category $D^b(\QS)$ is realized by a homomorphism of the complexes, so
$$\Hom_{D^b(\QS)}(W,W') = \Hom_{Kom(\QS)}(W,W').$$ 

\end{lemma}

\begin{proof}
The proof of each statement are nearly identical to that of Lemma 5.5 in \cite{CHW:CHW14} after switching $\QS$ in place of $\PP^2$ and noting that $\Hom\left(A^a\osum B^b,A^a \osum B^b\right) \isom \textnormal{GL}(a) \times \textnormal{GL}(b) \times M_{b,a}(\Hom(A,B))$ where $\{A,B\}$ is an exceptional pair and where $M_{b,a}(\Hom(A,B))$ is the group of $b$ by $a$ matrices with entries in $\Hom(A,B)$.
\end{proof}



Let $\{E_\alpha, E_\beta \}$ be an extremal controlling pair to $\xi$, $\{F_{-1},F_0\}$ be the left mutation of the minimally ranked right completion pair of $\{E_\alpha, E_\beta\}$, $\xi^+$ the primary orthogonal Chern character associated to that exceptional pair $\{E_\alpha, E_\beta \}$, and $U \in M(\xi)$ be a general element.
For simplicity, for the rest of the paper, we assume that $\chi(E_\beta^*,U)\geq 0$ and $\chi(E_\alpha^*,U)\leq 0$; proving the other cases is similar.
In Section \ref{sec: Beilinson Spec seq}, we saw that $U$ has a resolution of the form
$$0 \to E_\alpha^*(K)^{m_3} \osum F_0^{*m_2} \to F_{-1}^{*m_1} \osum E_\beta^{*m_0} \to U \to 0.$$
Using the resolution, we construct dominant rational maps from $M(\xi)$ to different Kronecker moduli spaces.

Which Kronecker moduli space we use, as well as the behavior of the map, depends on which, if any, of the $m_i$ are zero.
At most two of the $m_i$ are zero because $\xi$ is not the Chern class of $E^f$ for any exceptional $E$.
We break up the cases by the number of $m_i$ which are zero.

If no $m_i$ is zero, we construct a dominant rational map from $M(\xi)$ to $Kr_N(m_2,m_1)$ where $N = \hom(F_0^*,F_{-1}^*)$.

If exactly one $m_i$ is zero, then we could construct a dominant rational map to a certain Kronecker moduli space, but that space would always be a single point as one part of the dimension vector would be $0$. 
The constant map tells us nothing about our space, so we will not construct it here.

If $m_{i_0}$ and $m_{i_1}$ are zero, we construct a dominant rational map from $M(\xi)$ to $Kr_N(m_{j_1},m_{j_0})$ where $N$ is the dimension of the appropriate group of homomorphisms and $\{ i_0,i_1,j_0,j_1\}$ is some permutation of the set $\{0,1,2,3,\}$ and $j_1<j_0$.

\subsection{The Case When Two Powers Vanish}
\label{subsec:birational fib}
First note that the cases $m_3=m_2=0$ and $m_1=m_0=0$ cannot occur due to the form of the spectral sequence.
This leaves four cases where two exponents vanish to deal with. 
The proof of the proposition is identical in each case after you replace the two bundles with nonzero exponents so we will only explicitly prove the proposition in the first case.

\subsubsection{When the Second and Third Powers Vanish} 
For this subsubsection, assume that $m_2 = 0$ and $m_1 = 0 $ in the resolution
$$0 \to E_\alpha^*(K)^{m_3} \osum F_0^{*m_2} \fto{\phi} F_{-1}^{*m_1} \osum E_\beta^{*m_0} \to U \to 0.$$
We see that $U$ determines a complex of the form
$$W: E_\alpha^*(K)^{m_3} \to E_{\beta}^{*m_0}$$
which in turn determines a Kronecker $\Hom(E_\alpha^*(K),E_\beta^*)$-module.

Conversely, given a general such module and its determined complex, $W'$, there exists an element $U' \in M(\xi)$ such that $W'$ is the associated complex of $U'$ by Thm. \ref{thm:resolution-neg pos}. 
Assuming that the Kronecker module associated to a general $W$ is semistable, this constructs a rational map 
$$\pi : M(\xi) \dashrightarrow Kr_N(m_3,m_0)$$
where $N =\hom(E_{\alpha}^*(K),E_\beta^*)$.
In order to show that the Kronecker module associated to a general $W$ is semistable, it suffices by Subsection \ref{subsec:kronecker modules} to show that $Kr_N(m_3,m_0)$ is positive dimensional.

\begin{prop}
\label{prop:Kron birational fib}
With the above notation, $Kr_N(m_3,m_0)$ is positive dimensional, and the dominant rational map 
$$\pi : M(\xi) \dashrightarrow Kr_N(m_3,m_0)$$
is a birational map.
\end{prop}

\begin{proof}
By construction, the primary orthogonal Chern characters to $\xi$ are all semistable so $\xi^+$ is semistable.
By assumption, the general $U\in M(\xi)$ has the resolution 
$$0 \to E_\alpha^*(K)^{m_3} \to E_\beta^{*m_0} \to U \to  0.$$
A general map in a complex of the form 
$$W: E_\alpha^*(K)^{m_3} \to E_\beta^{*m_0}$$
is injective and has a semistable cokernel with Chern character $\xi$ by Thm. \ref{thm:resolution-neg pos}.
We also know that any isomorphism of two general elements of $M(\xi)$ is induced by an isomorphism of their resolutions. 

Recall from earlier in the section that $W$ corresponds to a Kronecker $\hom(E_\alpha^*(K),E_\beta^*)$-module $e$ with dimension vector $(m_3,m_0)$.
Then we compute that 
$$\dim(M(\xi)) = 1 - \chi(U,U) = 1- \chi(e,e) = (\textnormal{edim}) Kr_N(m_3,m_0).$$
As $\dim(M(\xi))>0$, we have that $(\textnormal{edim}) Kr_N(m_3,m_0)>0$.
By the properties of Kronecker moduli spaces, $Kr_N(m_3,m_0)$ is positive dimensional.
Thus, the general such module is stable.
As isomorphism of complexes corresponds exactly with isomorphism of Kronecker modules, we obtain a birational map 
$$\pi: M(\xi) \dashrightarrow Kr_N(m_3,m_0).$$
\end{proof}


\subsection{When All of the Powers Are Nonzero}
\label{subsec:large fib}
For this subsection, assume that $m_i \neq 0$ for all $i$ in the resolution
$$0 \to E_\alpha^*(K)^{m_3} \osum F_0^{*m_2} \fto{\phi} F_{-1}^{*m_1} \osum E_\beta^{*m_0} \to U \to 0.$$
Forgetting most of the information of the resolution, $U$ determines a complex of the form
$$W: F_0^{*m_2} \to F_{-1}^{*m_1}$$
which in turn determines a Kronecker $\Hom(F_0^*,F_{-1}^*)$-module.

Conversely, given a general such module and its determined complex, $W'$, there exists an element $U' \in M(\xi)$ such that $W'$ is the associated complex of $U'$ by Thm. \ref{thm:resolution-neg pos}.
Assuming that there is a semistable Kronecker module associated to a general $W$, this constructs a rational map 
$$\pi : M(\xi) \dashrightarrow Kr_N(m_2,m_1)$$
where $N =\hom(F_0^*,F_{-1}^*)$.
In order to verify that we get Kronecker modules, we have to prove a result about the map $E_{\beta}^{*m_0} \to U$ and $U \to E_\alpha^{*m_3}(K)[1]$ being canonical. 
To show that the general associated Kronecker module is semistable, we will again have to show that the Kronecker moduli space is nonempty.
We now show that the map $E_{\beta}^{*m_0} \to U$ is canonical.

\begin{proposition}\label{triangleThm}
With the notation of this subsection, let $U\in M(\xi)$ be general.  Let $W' \in D^b(\QS)$ be the mapping cone of the canonical evaluation map $$E_{\beta}^*\otimes \Hom(E_{\beta}^*,U) \to U,$$ so that there is a distinguished triangle 
$$E_{\beta}^* \otimes \Hom(E_{\beta}^*,U) \to U \to W' \to \cdot.$$ 
Then $W'$ is isomorphic to a complex of the form 
$$\left(E_{\alpha}^*(K) \otimes \CC^{m_3}\right) \osum \left(F_0^* \otimes \CC^{m_2}\right) \to \left(F_{-1}^* \otimes \CC^{m_1}\right)$$ 
sitting in degrees $-1$ and $0$.  

Furthermore, $W'$ is also isomorphic to the complex $\left(E_{\alpha}^*(K) \otimes \CC^{m_3}\right) \osum \left(F_0^* \otimes \CC^{m_2}\right) \to \left(F_{-1}^* \otimes \CC^{m_1}\right)$ appearing in the Beilinson spectral sequence for $U$.
\end{proposition}

\begin{proof}
It is easy to show that if $$0\to A\to B\oplus C \to D\to 0$$ is an exact sequence of sheaves, then the mapping cone of $C\to D$ is isomorphic to the complex $A\to B$ sitting in degrees $-1$ and $0$. 

Suppose we have a resolution $$0\to E_{\alpha}^*(K)^{m_3}  \osum F_0^{*m_2}\fto{\phi} F_{-1}^{*m_1} \osum E_{\beta}^{*m_0} \to U\to 0$$ of a general $U$ as in Thm. \ref{thm:resolution-neg pos}.  
We have $m_0 = \hom(E_{\beta}^*,U)$. 
Then since $U$ is semistable, the map $E_{\beta}^{*m_0} \to U$ can be identified with the canonical evaluation $E_{\beta}^* \otimes \Hom(E_{\beta}^*,U)\to U$.  
Thus, the mapping cone of this evaluation is the complex given by the first component of $\phi$.  
By Lemma \ref{lem:complexes}, any two complexes of the form $E_{\alpha}^*(K)^{m_3} \osum F_0^{*m_2} \to F_{-1}^{m_1}$ which are isomorphic in the derived category are in the same orbit of the $\textnormal{GL}(m_3)\times \textnormal{GL}(m_2)\times M_{m_2,m_3}(\Hom(E_{\alpha}^*(K),F_0^*)) \times \textnormal{GL}(m_1)$ action.

Finally, we show that $W'$ is isomorphic to the complex which appears in the Beilinson spectral sequence for $U$.  
For simplicity, we assume we are in the case $\Delta_1=\Delta_2 = 1$.
We recall how to compute the $E_1^{p,q}$-page of the spectral sequence.  Let $p_i:\left(\QS\right)\times\left( \QS\right)\to \QS$ be the projections, and let $\Delta\subset \left( \QS \right) \times \left(\QS\right)$ be the diagonal.  
There is a resolution of the diagonal 
$$0\to E_{\alpha}^*(K) \boxtimes  E_\alpha \osum F_0^* \boxtimes E_0 \to F_{-1}^* \boxtimes E_1 \to E_{\beta}^* \boxtimes E_{\beta} \to \OO_\Delta \to 0.$$  

We split the resolution of the diagonal into two short exact sequences 
$$0\to E_{\alpha}^*(K) \boxtimes  E_\alpha \osum F_0^* \boxtimes E_0 \to F_{-1}^* \boxtimes E_1 \to M\to 0 \textnormal{ and}$$
$$0\to M\to E_{\beta}^* \boxtimes E_{\beta} \to \OO_{\Delta}\to 0.$$ 
Tensoring with $p_2^*(U)$ and applying $R{p_{1*}}$, we get triangles 
$$\Phi_{E_{\alpha}^*(K) \boxtimes E_\alpha}(U)\osum \Phi_{F_0^* \boxtimes E_0}(U) \to \Phi_{F_{-1}^* \boxtimes E_1}(U)\to \Phi_M(U)\to\cdot \textnormal{ and}$$
$$\Phi_M(U)\to \Phi_{E_{\beta}^*\boxtimes E_{\beta}}(U)\to \Phi_{\OO_{\Delta}}(U)\to \cdot,$$ 
where $\Phi_F:D^b(\QS)\to D^b(\QS)$ is the Fourier-Mukai transform with kernel $F$.  
Computing these transforms using Proposition \ref{thm:resolution-neg pos}, we obtain two different complexes 
$$ E_{\alpha}^*(K)\otimes H^1(E_\alpha \otimes U) \osum F_0^* \otimes H^0(E_0 \otimes U) \to F_{-1}^* \otimes H^0(E_1 \otimes U) \to \Phi_M(U)[1]\to \cdot \textnormal{ and}$$ 
$$\Phi_M(U) \to E_{\beta}^* \otimes \Hom(E_{\beta}^*,U) \to  U\to \cdot$$
involving $\Phi_M(U)$; notice that the map $E_{\beta}^*\otimes \Hom(E_{\beta}^*,U)\to U$ is the canonical one since the map $$E_{\beta}^* \boxtimes E_{\beta} \to \OO_\Delta$$ is the trace map.
Therefore $\Phi_M(U)[1]$ is isomorphic to $W$ by the second triangle.  
On the other hand, $\Phi_M(U)[1]$ is also isomorphic to the complex in the Beilinson spectral sequence by the first triangle.
\end{proof}

We now turn to showing that the map $U \to E_{\alpha}^*(K)^{m_0}[1]$ is canonical in order to show that the general $W$ is associated to a Kronecker module.
%
%
%
%
%

\begin{proposition}\label{triangleThm otherway}
With the notation of this subsection, let $U\in M(\xi)$ be general.  Let $W' \in D^b(\QS)$ be the mapping cone of the canonical evaluation map $$U \to E_{\alpha}^*(K)[1]\otimes \Hom(U,E_{\alpha}^*(K)[1]),$$ so that there is a distinguished triangle 
$$U \to E_{\alpha}^*(K)[1] \otimes \Hom(U,E_{\alpha}^*(K)[1]) \to W' \to \cdot.$$ 
Then $W'$ is isomorphic to a complex of the form 
$$ \left(F_0^* \otimes \CC^{m_2}\right) \to \left(F_{-1}^* \otimes \CC^{m_1}\right) \osum \left(E_{\beta}^* \otimes \CC^{m_0}\right)$$ 
sitting in degrees $-2$ and $-1$.  

Furthermore, $W'$ is also isomorphic to the complex $\left(F_0^* \otimes \CC^{m_2}\right) \to \left(F_{-1}^* \otimes \CC^{m_1}\right) \osum \left(E_{\beta}^* \otimes \CC^{m_0}\right) $ appearing in the Beilinson spectral sequence for $U$.
\end{proposition}

\begin{proof}
It is easy to show that if $$0\to A\oplus B \to C \to D\to 0$$ is an exact sequence of sheaves, then the mapping cone of $D\to A[1]$ is isomorphic to the complex $B\to C$ sitting in degrees $-2$ and $-1$.  
Suppose we have a resolution $$0\to E_{\alpha}^*(K)^{m_3}  \osum F_0^{*m_2}\fto{\phi} F_{-1}^{*m_1} \osum E_{\beta}^{*m_0} \to U\to 0$$ of a general $U$ as in Thm. \ref{thm:resolution-neg pos}.  
We have $m_3 = \hom(U,E_{\alpha}^*(K)[1])$.
Then since $U$ is semistable, the map $U \to E_{\alpha}^*(K)^{m_3}[1]$ can be identified with the canonical co-evaluation $U \to E_{\alpha}^*(K)[1] \otimes \Hom(U,E_{\alpha}^*(K)[1])$.  
Thus, the mapping cone of this co-evaluation is the complex given by the second component of $\phi$.  
By Lemma \ref{lem:complexes}, any two complexes of the form $F_0^{*m_2} \to F_{-1}^{m_1} \osum E_{\beta}^{*m_0}$ which are isomorphic in the derived category are in the same orbit of the $\textnormal{GL}(m_2)\times \textnormal{GL}(m_1)\times \textnormal{GL}(m_0) \times M_{m_0,m_1}(\Hom(F_{-1}^*,E_{\beta}^*))$ action.

Finally, we show that $W'$ is isomorphic to the complex which appears in the Beilinson spectral sequence for $U$.  
For simplicity, we assume we are in the case $\Delta_1=\Delta_2 = 0$.
Recall how to compute the $E_1^{p,q}$-page of the spectral sequence.  Let $p_i:\left(\QS\right)\times\left( \QS\right)\to \QS$ be the projections, and let $\Delta\subset \left( \QS \right) \times \left(\QS\right)$ be the diagonal.  
There is a resolution of the diagonal 
$$0\to E_{\alpha}^*(K) \boxtimes  E_\alpha \to F_0^* \boxtimes E_0 \to F_{-1}^* \boxtimes E_1 \osum E_{\beta}^* \boxtimes E_{\beta} \to \OO_\Delta \to 0.$$  

We split the resolution of the diagonal into two short exact sequences 
$$0 \to F_0^* \boxtimes E_0 \to F_{-1}^* \boxtimes E_1 \osum E_{\beta}^* \boxtimes E_{\beta} \to M\to 0 \textnormal{ and}$$
$$0\to M \to \OO_{\Delta}\to  \left( E_{\alpha}^*(K) \boxtimes  E_\alpha \right)[1].$$
Tensoring with $p_2^*(U)$ and applying $R{p_{1*}}$, we get triangles 
$$\Phi_{F_0^* \boxtimes E_0}(U) \to \Phi_{F_{-1}^* \boxtimes E_1}(U) \osum \Phi_{E_{\beta}^*\boxtimes E_{\beta}}(U) \to \Phi_M(U)\to\cdot \textnormal{ and}$$
$$\Phi_M(U)\to \Phi_{\OO_{\Delta}}(U)\to \Phi_{(E_{\alpha}^*(K) \boxtimes E_\alpha)[1]}(U) \to \cdot,$$ 
where $\Phi_F:D^b(\QS)\to D^b(\QS)$ is the Fourier-Mukai transform with kernel $F$.  
Computing these transforms using Thm. \ref{thm:resolution-neg pos}, we obtain two different complexes 
$$ F_0^* \otimes H^1(E_0 \otimes U) \to F_{-1}^* \otimes H^1(E_1 \otimes U) \osum E_{\beta}^* \otimes \Hom(E_{\beta}^*,U) \to \Phi_M(U)\to \cdot \textnormal{ and}$$ 
$$\Phi_M(U) \to  U\to  E_{\alpha}^*(K)\otimes \Hom(U,E_\alpha^*(K)[1]) \to \cdot$$
involving $\Phi_M(U)$; notice that the map $U \to E_{\alpha}^*(K)[1]\otimes \Hom(U,E_{\alpha}^*(K)[1])$ is the canonical one since the map $$\OO_\Delta \to (E_{\alpha}^*(K) \boxtimes E_{\alpha})[1]$$ is the cotrace map.
Therefore, $\Phi_M(U)$ is isomorphic to $W'$ by the second triangle.  
On the other hand, $\Phi_M(U)$ is also isomorphic to the complex in the Beilinson spectral sequence by the first triangle.
\end{proof}

We now use that these maps are canonical to establish that each resolution is associated to a Kronecker module.
\begin{prop}\label{prop: Kron}
Let $U \in M(\xi)$ be a general object with the resolution 
$$0 \to E_{\alpha}^*(K)^{m_3} \osum F_{0}^{*m_2} \to F_{-1}^{*m_1} \osum E_{\beta}^{*m_0} \to U \to 0$$
with subcomplex
$$W: F_{0}^{*m_2} \to F_{-1}^{*m_1}.$$
Then $W$ is the complex appearing in the spectral sequence.
\end{prop}

\begin{proof}
Let $W': F_{0}^{*m_2} \to F_{-1}^{*m_1}$ be the complex appearing in $U$' spectral sequence.
By a Bertini-like statement \cite{Hu:Hu16}, $W'$ is either surjective or injective.

Assume that it is injective.
This means that the $E_2$ page of the spectral sequence is
$$
\begin{array}{cccc}
E_{\alpha}^*(K)^{m_3} & 0 & 0 & 0 \\
0 & 0 & K & 0 \\
0 & 0 & 0 & E_{\beta}^{*m_0} \\
\end{array}
$$
where $K = \textnormal{coker}(W')$.
In turn this gives that, in the resolution coming from the spectral sequence, the map from $F_{0}^{*m_2}$ to $E_{\beta}^{*m_0}$ is zero.
Then, by Prop. \ref{triangleThm otherway}, we have two short exact sequences
$$0 \to F_{0}^{*m_2} \fto{\phi} F_{-1}^{*m_1} \osum E_{\beta}^{*m_0} \to L \to 0\textnormal{ and}$$
$$0 \to F_{0}^{*m_2} \fto{\psi} F_{-1}^{*m_1} \osum E_{\beta}^{*m_0} \to L \to 0$$
where $W$ and $W'$ are subcomplexes of the respective sequences and the second component of $\psi$ is zero.
The identity map on $L$ induces an isomorphism of its resolutions, which implies that the the first component of $\phi$ is a scalar multiple of the first component of $\psi$.
This then gives an isomorphism between $W$ and $W'$ given by dividing by that scalar multiple in the degree $-1$ and the identity in degree $0$.
Thus, $W$ is the complex in the spectral sequence converging to $U$.

The case of surjectivity is similar but uses Prop. \ref{triangleThm}.
\end{proof}

We now finish the construction of the map $M(\xi) \dashrightarrow Kr_N(m_2,m_1)$.
\begin{prop}
\label{prop:Kron large fib}
With the above notation, $Kr_N(m_2,m_1)$ is nonempty and there is a dominant rational map 
$$\pi : M(\xi) \dashrightarrow Kr_N(m_2,m_1).$$
\end{prop}

\begin{proof}
By construction, we know that the primary orthogonal Chern characters to $\xi$ are all semistable so $\xi^+$ is semistable.
Let $V \in M(\xi^+)$ be general.
Then, by Thm. \ref{thm:resolution-neg pos},
$V$ has the resolution 
$$0 \to E_1^{n_1} \to V \to E_0^{n_2} \to 0$$
or one of the equivalent resolutions which have the same Kronecker module structure.
Similarly, the general $U\in M(\xi)$ has the resolution
$$E_\beta^{*m_0} \to U \to \left(E_\alpha^*(K)^{m_3}\osum W\right)[1] $$
where $W$ is the complex
$$F_0^{*m_2} \to F_{-1}^{*m_1}.$$

As the point $(\mu^+,\Delta^+)$ lies on the surface $Q_\xi$, we have that $\chi(V^*,U)= 0$.
By design, the resolution of $V$ immediately forces $\chi(V^*, E_\beta^*) =0$ and $\chi(V^*, E_\alpha^*(K)) =0$ because of the orthogonality properties of the coil $\{E_\alpha^*(K),E_\beta^*,E_1^*,E_0^*\}$.
Then vanishings of $\chi(V^*,U)$, $\chi(V^*, E_\beta^*)$, and $\chi(V^*, E_\alpha^*(K))$ force $\chi(V^*,W)=0$.
Since $\chi(V^*,W)=0$ and $\chi(E_\beta^*,W) =0$, we have that $\chi(L_{E_\beta^*} V^*,W) = 0$.
Shifting only shifts the indices so we have $\chi(L_{E_\beta^*} V^*[1],W) = 0$ as well.
In the derived category $L_{E_\beta^*} V^*[1]$ is isomorphic to the complex 
$$F_0^{*n_1} \to F_{-1}^{*n_2}$$
sitting in degrees -1 and 0.
Thus, $L_{E_\beta^*} V^*[1]$ and $W$ both correspond to Kronecker $\Hom(F_0^*,F_1^*)$-modules.
Call them $e$ and $f$ respectively.

Then $\chi(L_{E_\beta^*} V^*[1],W) = 0$ tells us that $\chi(e,f) = 0$ which implies that $\underline{\dim} \textnormal{ } f$ if a right-orthogonal dimension vector to $\underline{\dim} \textnormal{ }e$.
Since $M(V)$ is nonempty, Prop. \ref{prop:Kron birational fib} shows that $Kr_N(n_1,n_2)$ is nonempty.
If $Kr_N(n_1,n_2)$ is positive(0) dimensional, the discussion at the end of Subsection $6.1$ of \cite{CHW:CHW14} shows that $Kr_N(m_2,m_1)$ is as well.
Thus, $Kr_N(m_2,m_1)$ is nonempty as promised.

\end{proof}


\section{Primary extremal rays of the effective cone}
\label{sec: primary extremal rays of the effective cone}
In this section, we use the maps from $\MMM$ to Kronecker moduli spaces that we constructed in the previous section to give an alternate description of effective Brill-Noether divisors and to show that they are extremal.
Let $\xi^+$ a primary orthogonal Chern character to $\MMM$ with $V \in M(\xi^+)$ general.
The way in which we can express the Brill-Noether divisor $D_V$ depends greatly on the dimension of the Kronecker moduli space, $K$, that we map to (as dictated by the previous section). 

If $\dim(K) = 0$ (or a single $m_i$ is zero so we did not construct a fibration), then the Kronecker fibration is a map to a point so it does not give us any information so we do not use it at all. 
In this case, $V$ is an exceptional bundle.
This divisor consists exactly of those elements in $\MMM$ without the specified resolution, and the dual moving curve(s) are found by varying the maps in the resolution.

If $\dim(K) > 0$, then the Kronecker fibration is far more interesting.
In this case, $\xi^+$ may or may not be exceptional and the Brill-Noether divisor $D_V$ is the indeterminancy or exceptional locus of the map from $\MMM$ to the Kronecker moduli space.
Either this map is birational, in which case the moving curve is gotten by varying the Kronecker module, or the map has positive dimensional fibers, in which case the moving curve(s) are gotten by varying the other maps in the resolution to cover the fibers of the map.
If certain numeric inequalities hold, there are two dual moving curves covering the (positive dimensional) fibers of the map which implies that $D_V$ is the pullback of a generator of the ample cone of the Kronecker moduli space; in the case, the Brill-Noether divisor $D_V$ is also inside the movable cone.

Let $\{E_\alpha, E_\beta \}$ be an associated extremal pair to $\xi$ with orthogonal Chern character $\zeta$, $\{F_{-1},F_0\}$ be the left mutation of the minimally ranked right completion pair of $\{E_\alpha, E_\beta\}$, $U \in M(\xi)$ be a general element, and $K$ be the Kronecker moduli space containing the Kronecker module appearing in the resolution of $U$.

\subsection{The Zero Dimensional Kronecker Moduli Space Case}
\label{subsec:eff cone=  =}
\begin{thm}
\label{thm:eff cone= =}
Let $\xi^+$ be a primary orthogonal Chern character to $\{\alpha, \beta\}$ for the Chern character $\xi$ with $\dim(K)=0$ and let $V \in M(\xi^+)$ be the element.
Then the Brill-Noether divisor $$D_{V} = \{U' \in M(\xi) : h^1(U' \otimes V) \neq 0 \}$$
is on an edge of the effective cone of $M(\xi)$.
Using the isomorphism $\ns\left( M(\xi) \right) \isom \xi^\perp$, $D_{V}$ corresponds to $\xi^+$.
\end{thm}

\begin{proof}
The general element $U \in M(\xi)$ fits into the short exact sequence
$$0 \to E_\alpha^*(K)^{m_3}\osum F_0^{*m_2} \to  F_{-1}^{*m_1} \osum \left( E_\beta^*\right)^{m_0} \to U \to 0 $$
In order to show that $D_V$ is an effective Brill-Noether divisor, we have to show that $V$ is cohomologically orthogonal to $U$.
This means showing that $U \otimes V = \sheafhom(U^*,V)$ has no cohomology.
How we show this orthogonality depends upon which if any of the $m_i$ vanish.
Note that $\dim(K)=0$ implies that at most one of the $m_i$ is zero because if two are zero then $\dim(K) = \dim(M(\xi))$.

Assume that none of the $m_i$ are zero.
Then the general element, $U \in M(\xi)$, fits into the triangle
$$\left( E_\beta^*\right)^{m_0} \to U \to E_\alpha^*(K)^{m_3}[1] \osum W$$
where $W$ is the complex $\left(F_0^*\right)^{m_2} \to \left(F_{-1}^*\right)^{m_1}$ sitting in degrees -1 and 0.
Similarly, the general element, $V \in M(\xi^+)$, fits into the triangle
$$E_1^{n_2} \to V \to E_0^{n_1}.$$
By choice of resolving exceptional bundles, $\sheafhom(E_\beta^*,V)$ and $\sheafhom(E_\alpha^*(K),V) = 0$ have no cohomology.
Thus, to construct the divisor it suffices to show that $\sheafhom(W^*,V)$ has no cohomology.
As $\sheafhom(W^*,E_\beta)$ has no cohomology, this is equivalent to $\sheafhom(W^*,R_\beta V)$ having no cohomology.
We reduce further to showing that $\sheafhom(W^*,R_\beta V[1])$ has no cohomology as shifting merely shifts the cohomology.
Then if $f$ and $e$ are the Kronecker modules corresponding to $W^*$ and $R_\beta V[1]$, respectively, the vanishing of these cohomologies is equivalent to the vanishing of the $\Hom(f,e)$, but that vanishing follows directly from Thm. 6.1 of \cite{CHW:CHW14}. 
Thus, we have the orthogonality that we needed.

If one of the $m_i$ is zero (in which case we have not constructed a Kronecker fibration explicitly), then $V$ is one of the exceptional bundles $E_\beta$, $E_1$, $E_0$, or $E_\alpha$.
Then $V$ is cohomologically orthogonal to all three bundles that appear in the resolution of $U$ so it is automatically cohomologically orthogonal.

Thus, we have shown the cohomological orthogonality in either case.
The class of $D_V$ and the fact that it is effective is computed using Prop. \ref{prop:brillnoether}.

To show it lies on an edge, we construct a moving curve by varying a map in the resolution.

If $m_3 \neq 0$ and $m_0 \neq 0$, fix 
every map except $E_\alpha^*(K)^{m_3} \to \left( E_\beta^*\right)^{m_0}$, let
$$S= \PP \Hom \left(E_\alpha^*(K)^{m_3},\left( E_\beta^*\right)^{m_0}  \right), $$
and let $\mathcal{U} / S$ be the universal cokernel sheaf (of the fixed map plus the varying part).

If $m_3 = 0$, fix 
every map except $F_{0}^{*m_2} \to \left( E_\beta^*\right)^{m_0}$, let
$$S= \PP \Hom \left(F_0^{*m_2},\left( E_\beta^*\right)^{m_0}  \right), $$
and let $\mathcal{U} / S$ be the universal cokernel sheaf (of the fixed map plus the varying part).

If $m_0 = 0$, fix 
every map except $E_\alpha^*(K)^{m_3} \to \left( F_0^*\right)^{m_1},$ let
$$S= \PP \Hom \left(E_\alpha^*(K)^{m_3},\left( F_0^*\right)^{m_1}  \right), $$
and let $\mathcal{U} / S$ be the universal cokernel sheaf (of the fixed map plus the varying part).

In any case, we have our $\mathcal{U}$ and our $S$.
Because $\MMM$ is positive dimensional and the general sheaf in it has a resolution of this form, $S$ is nonempty.
Then $\mathcal{U}$ is a complete family of prioritary sheaves whose fixed Chern character lies above the $\delta$ surface.
Define the open set $S' \subset S$ by $$S' := \{s\in S: \mathcal{U}_s \textnormal{ is stable} \}$$ 
Thus, by assumption, the complement of $S'$ has codimension at least 2 which allows us to find a complete curve in $S'$ containing the point corresponding to $U$ for the general element $U\in M(\xi)$.
Notice that this is a moving curve by the codimension statement.
Any curve in $S'$ is disjoint from $D_{V}$ which makes the curve dual to it.

This curve makes $D_{V}$ be on an edge. 
As the resolution only provides one moving curve, this resolution only shows that it lies on an edge of the cone, not that it is an extremal ray.
\end{proof}


\subsection{The Positive Dimensional Kronecker Moduli Space Case}

\begin{thm}
\label{thm:eff cone neq neq}
Let $\xi^+$ be a primary orthogonal Chern character to $\{\alpha, \beta\}$ for the Chern character $\xi$ with $\dim(K)>0$ and $V \in M(\xi^+)$ be a general element.
Then the Brill-Noether divisor $$D_V = \{U' \in M(\xi) : h^1(U' \otimes V) \neq 0 \}$$
lies on the edge of the effective cone of $M(\xi)$.
Using the isomorphism $\ns\left( M(\xi) \right) \isom \xi^\perp$, $D_V$ corresponds to $\xi^+$.
\end{thm}

\begin{proof}
Recall that we have a dominant rational map $\pi: \MMM \dashrightarrow K$, 
There are two possibilities; either $\pi$ is a birational map or $\pi$ has positive dimensional fibers.

\textit{Birational Case}
In this case, either zero or two of the $m_i$ can vanish.
If zero vanish, we show that $V$ is cohomologically orthogonal to $U$ by the same arguments as the previous theorem.
If two vanish, then $V$ is again one of the exceptional bundles so orthogonality is immediate.
The class of $D_V$ and the fact that it is effective is computed using Prop. \ref{prop:brillnoether}, and $D_V$ is the exceptional locus of $\pi$.
Using that fact, we get a dual moving curve to $D_V$ by varying the Kronecker module.
Formally, because $K$ is Picard rank one, there is a moving curve $C$.
Then $[\pi^*(C)]$ is a moving curve which is dual to the exceptional locus of $\pi$ (i.e. dual to $D_V$).
Thus, $D_V$ is on the edge of the effective cone.

\textit{Positive dimensional fiber case}
In this case, none of the $m_i$ is zero.

We first show cohomological orthogonality. 
This means showing that $U \otimes V = \sheafhom(U^*,V)$ has no cohomology.
In this case, the general element, $U \in M(\xi)$, fits into the triangle
$$\left( E_\beta^*\right)^{m_0} \to U \to E_\alpha^*(K)^{m_3}[1] \osum W$$
where $W$ is the complex $\left(F_0^*\right)^{m_2} \to \left(F_{-1}^*\right)^{m_1}$ sitting in degrees -1 and 0.
Similarly, the general element, $V \in M(\xi^+)$, fits into the triangle
$$E_1^{n_2} \to V \to E_0^{n_1}.$$
By choice of resolving exceptional bundles, $\sheafhom(E_\beta^*,V)$ and $\sheafhom(E_\alpha^*(K),V)$ have no cohomology.
Thus, to construct the divisor it suffices to show that $\sheafhom(W^*,V)$ has no cohomology.
As $\sheafhom(W^*,E_\beta)$ has no cohomology, we have that this is equivalent to $\sheafhom(W^*,R_\beta V)$ having no cohomology.
We reduce further to showing that $\sheafhom(W^*,R_\beta V[1])$ has no cohomology as shifting merely shifts the cohomology.
Then if $f$ and $e$ are the Kronecker modules corresponding to $W^*$ and $R_\beta V[1]$, respectively, the vanishing of these cohomologies is equivalent to the vanishing of the $\Hom(f,e)$, but that vanishing follows directly from Thm. 6.1 of \cite{CHW:CHW14} 


We have now established the cohomological orthogonality.
The class of $D_V$ and the fact that it gives an effective divisor are computed using Prop. \ref{prop:brillnoether}.

As the general $U$ has the given resolution, the fibers of map to $K$ are covered by varying the other maps of the resolution as we did in the last proof. 
As these moving curves sit inside fibers, they are dual to $D_V$ since $V$ is dual to the Kronecker modules in $K$.

If we can vary two different maps in the resolution other than the Kronecker module independently, than $D_V$ has the same class as the pullback of an ample divisor of $K$.
This immediately implies that $D_V$ is in the moving cone.
We can vary two maps independently if the no $m_i$ is zero and no subcomplex of the resolution has enough dimensions to account for all of the dimensions of our moduli space.
\end{proof}

These theorems together give an effective divisor on $\MMM$.
These conjecturally might give a spanning set of effective divisors for the effective cone of $\MMM$, but even if this method does not do that, it gives a way to construct effective divisors on many of these moduli spaces.
In addition, this same method works to give secondary extremal rays when the rank of the moduli space is at least three (for rank less than three all secondary rays have special meaning).

\section{Examples}
\label{sec: examples}

The method of the previous section constructs Brill-Noether divisors on the faces of the effective cone for moduli spaces of sheaves on $\QS$. 
In this section, we work out a series of examples showing the usefulness of our theorems.
We work out the effective cones of the first fifteen Hilbert schemes of points as well as some series of extremal rays that occur for infinitely many Hilbert schemes of points on $\QS$.
Lastly, we provide an extremal edge for the effective cone of a moduli space of rank two sheaves with nonsymmetric slope so that we see the theorems are useful in that setting as well.


\subsection{The Effective Cones of Hilbert Schemes of at Most Sixteen Points}
\label{subsec: n less than 17}

The most classical example of a moduli space of sheaves on $\QS$ is Hilbert scheme of $n$ points on it.
For these Hilbert schemes, the Picard group has a classical basis, $\{B, H_1,H_2\}$.
Each element of this basis has an extremely geometric interpretation.
$B$ is the locus of nonreduced schemes or equivalently the schemes supported on $n-1$ or fewer closed points.
$H_1$ is the schemes whose support intersects a fixed line of type $(1,0)$.
Similarly, $H_2$ is the schemes whose support intersects a fixed line of type $(0,1)$.

Using this basis, every ray in the N\'eron-Severi space is spanned by a ray of the form $B$, $aH_1+bH_2+B$, $aH_1 +bH_2$, or $iH_1+jH_2-\frac{B}{2}$.
Then we fix the notation for the last two types of ray as 
$$Y_{a,b} = aH_1+bH_2 \textnormal{  and }$$
$$X_{i,j} = iH_1 +j H_2 -\frac{1}{2}B.$$

Using this notation, we list the extremal rays of the effective cones of $\left(\QS\right)^{\left[n\right]}$ 
for $n\leq 16$, explicitly work out the case of $n=7$, prove that some sequences of rays are extremal for varying $n$, and then finally explicitly work out each remaining extremal ray for $n\leq 16$.
These are all new results except for the cases of $n\leq 5$.

\begin{center}
\begin{tabular}{|c|c|}
 \hline
 n & Extremal Rays  \\ 
 \hline
 2 & $B$, $X_{1,0}$, and $X_{0,1}$  \\ 
 \hline
 3 & $B$, $X_{2,0}$, and $X_{0,2}$ \\ 
 \hline
 4 & $B$, $X_{3,0}$, $X_{1,1}$, and $X_{0,3}$ \\ 
 \hline
 5 & $B$, $X_{4,0}$, $X_{\frac{4}{3},\frac{4}{3}}$, and $X_{0,4}$ \\ 
 \hline
 6 & $B$, $X_{5,0}$, $X_{2,1}$, $X_{1,2}$, and $X_{0,5}$ \\ 
 \hline
 7 & $B$, $X_{6,0}$, $X_{\frac{12}{5},\frac{6}{5}}$, $X_{2,\frac{3}{2}}$, $X_{\frac{3}{2},2}$, $X_{\frac{6}{5},\frac{12}{5}}$, and $X_{0,6}$ \\ 
 \hline
 8 & $B$, $X_{7,0}$, $X_{3,1}$, $X_{1,3}$, and $X_{0,7}$ \\ 
 \hline
 9 & $B$, $X_{8,0}$, $X_{\frac{24}{7},\frac{8}{7}}$, $X_{2,2}$, $X_{\frac{8}{7},\frac{24}{7}}$, and $X_{0,8}$ \\ 
 \hline
 10 & $B$, $X_{9,0}$, $X_{4,1}$, $X_{\frac{5}{2},2}$, $X_{2,\frac{5}{2}}$, $X_{1,4}$, and $X_{0,9}$ \\ 
 \hline
 11 & $B$, $X_{10,0}$, $X_{\frac{40}{9},\frac{10}{9}}$, $X_{4,\frac{4}{3}}$, $X_{\frac{12}{5},\frac{12}{5}}$, $X_{\frac{4}{3},4}$, $X_{\frac{10}{9},\frac{40}{9}}$,and $X_{0,10}$ \\ 
 \hline
 12 & $B$, $X_{11,0}$, $X_{5,1}$, $X_{3,2}$, $X_{2,3}$, $X_{1,5}$, and $X_{0,11}$ \\ 
 \hline
 13 & $B$, $X_{12,0}$, $X_{\frac{60}{11},\frac{12}{11}}$, $X_{\frac{9}{2},\frac{3}{2}}$, $X_{\frac{7}{2},2}$, $X_{\frac{8}{3},\frac{8}{3}}$, $X_{2,\frac{7}{2}}$, $X_{\frac{3}{2},\frac{9}{2}}$, $X_{\frac{12}{11},\frac{60}{11}}$, and $X_{0,12}$ \\ 
 \hline
 14 & $B$, $X_{13,0}$, $X_{6,1}$, $X_{\frac{10}{3},\frac{7}{3}}$, $X_{\frac{7}{3},\frac{10}{3}}$, $X_{1,6}$, and $X_{0,13}$ \\ 
 \hline
 15 & $B$, $X_{14,0}$, $X_{\frac{84}{13},\frac{14}{13}}$, $X_{4,2}$, $X_{2,4}$, $X_{\frac{14}{13},\frac{84}{13}}$, and $X_{0,14}$ \\ 
 \hline
 16 & $B$, $X_{15,0}$, $X_{7,1}$, $X_{\frac{9}{2},2}$, $X_{3,3}$, $X_{2,\frac{9}{2}}$, $X_{1,7}$, and $X_{0,15}$ \\ 
  
 \hline
\end{tabular}
\end{center}


\subsection{The Effective Cone of the Hilbert Scheme of $7$ Points}
\label{subsec: n equals 7}

It is worth showing how the theorem is applied in one these cases to compute the effective cone. 
Recall that the general strategy to compute an effective cone has two steps.
First, provide effective divisors.
Second, provide moving curves which are dual to the effective divisors.

We use our main theorem to do this for the primary extremal rays of the effective cone;
we have to deal with the secondary extremal rays separately. 
There is a single secondary extremal ray which is spanned by $B$.
$B$ is clearly an effective divisor as it is the locus of nonreduced schemes. 
In order to show that $B$ spans an extremal ray, we just have to construct two distinct dual moving curves.

We now construct these moving curves, $C_1$ and $C_2$.
We construct $C_1$ by fixing $6$ general points and then varying a seventh point along a curve of type $(1,0)$.
Similarly, we construct $C_2$ by fixing $6$ general points and then varying a seventh point along a curve of type $(0,1)$.
Any set of $7$ distinct points lies on at least one curve of type $C_1$, so it is a moving curve.
Similarly, $C_2$ is a moving curve.

We now show that $C_1$ and $C_2$ are dual to $B$.
Starting with six general points, we can find a line $l$ of type $(1,0)$ that does not contain any of the points.
We get a curve $C'$ of type $C_1$ in the Hilbert scheme by varying the seventh point along $l$.
As $l$ does not contain any of the six general points, every point in $C'$ corresponds to seven distinct points, so $C'$ does not intersect $B$.
Thus, we get
$$C_1 \cdot B = C' \cdot B = 0.$$ 
Similarly, we get that 
$$C_2 \cdot B = 0.$$

The only thing left to do in order to show that $B$ spans an extremal ray is to show that $C_1$ and $C_2$ have distinct classes.
Starting with six general points, we find lines $l$ and $l'$ of type $(1,0)$ that does not contain any of the points.
Again, we get a curve $C'$ of type $C_1$ in the Hilbert scheme by varying the seventh point along $l$.
Analogously, we get a divisor $H'$ of type $H_1$ as the locus of schemes whose support intersects $l'$.
As $l'$ does not contain any of the general fixed points and does not intersect $l$, we get that $H_1$ and $C'$ are disjoint.
Thus, 
$$C_1 \cdot H_1 = C' \cdot H' = 0.$$
Using the same six general points, we find a line $l_0$ of type $(0,1)$ that does not contain any of the points.
We get a curve $C_0$ of type $C_2$ by varying the seventh point along $l_0$.
Then $l_0$ does not contain any of the six general points by construction but does intersect $l'$ in exactly one point.
Thus, 
$$C_2 \cdot H_1 = C_0 \cdot H' =  1.$$
As $C_1 \cdot H_1 \neq C_2 \cdot H_1$, we know that $C_1$ and $C_2$ are distinct classes.
This observation completes the proof that $B$ spans an extremal ray.

While constructing the primary extremal rays, we will construct two moving curves dual to $B$.
These curves will show that $B$ is the only secondary extremal ray.
Also as the slope of the ideal sheaf is $(0,0)$, the effective cone is symmetric in the coordinates of $H_1$ and $H_2$ so we only deal with the primary rays spanned by $X_{i,j}$ where $i\geq j$. 
Keeping that in mind, we move to computing the primary extremal rays using our theorem.

One way to think about the main results of this paper are that they give an algorithm to compute the primary extremal rays of the effective cone of $\MMM$.
That algorithm breaks down roughly into four steps: find the extremal pairs, use the extremal pairs to resolve the general object of $\MMM$, use those resolutions to construct maps to moduli spaces of Kronecker modules, and analyze these maps to find divisors spanning extremal rays.
Let's follow those steps in this specific case.

\subsubsection*{Step 1}
As we noted above, the first step is to find all of the extremal pairs.
Proceed by finding all controlling exceptional bundles, finding the controlling pairs, and then finding which are extremal pairs.

Controlling exceptional bundles are those controlling the $\delta$-surface over the locus
$$\left\{X \in \left(1,\mu,\frac{1}{2}\right) \subset K(\QS): \chi(X\otimes \mathcal{I}_z) = 0 \textnormal{ for } \mathcal{I}_z \in \left(\QS\right)^{[n]}\right\}.$$
Using Mathematica, we find that these controlling exceptional bundles are 
$$\{\cdots, \{0, 4, 1\}, \{0, 5, 1\}, \{0, 6, 1\}, \{0, 7, 1\}, \{0, 8, 1\}, \{0, 9, 1\}, \{0, 10, 1\}, \{0, 11, 1\}, \{0, 12, 1\}, \{0, 13, 1\},$$
$$\{0, 14, 1\},
\{1, 27, 5\}, \{1, 13, 3\}, \{2, 11, 3\}, 
\{1, 2, 1\}, \{1, 3, 1\}, \{1, 4, 1\},  
\{6, 12, 5\}, 
\{2, 1, 1\}, \{2, 2, 1\},\cdots\}$$
where we record an exceptional bundle with Chern character $\left(r,(\mu_1,\mu_2),\Delta\right)$ as $(\mu_1,\mu_2,r)$ and we truncate the list when bundles can no longer possibly matter. 
We will see that the ones we have truncated do not matter as our first resolution will be dual to $B$.

There are many, many controlling exceptional pairs, but we do not need to see all of them.

Finally, we check to see which of these are extremal pairs.
They are whittled down by eliminating each pair that does not have each of the properties of an extremal pair.
The only four controlling pairs that are extremal pairs are
\newline
\begin{center}
$\{\OO(6,0),\OO(7,0)\}$, $\{\OO(3,1),\OO(6,0)\}$, $\{\OO(2,1),\OO(3,1)\}$, and $\{\OO(2,1),\OO(2,2)\}$.
\end{center}

Each extremal pair controls an extremal ray of the effective cone.
Recall that given an extremal pair $\{A,B\}$, the extremal ray it corresponds to is spanned by the primary orthogonal Chern character of the pair: $\ch(A)$, $\ch(B)$, or
$$p = \{ X \in K(\QS) : Q_{\xi,A}(X) = Q_{\xi,B} (X) = \chi(\mathcal{I}_z \otimes X) = 0\}.$$
Then the primary orthogonal Chern character for our exceptional pairs are $\left( 1,(6,0),0\right)$, $\left( 1,(6,0),0\right)$, $\left( 5,(12,6),12\right)$, and $\left( 2,(4,3),5\right)$, respectively.
These Chern characters correspond to the extremal rays $X_{6,0}$, $X_{6,0}$, $X_{\frac{12}{5},\frac{6}{5}}$, and $X_{2,\frac{3}{2}}$, respectively.  
Notice that one of the rays is repeated twice.
This repetition is because we need all of these extremal pairs to share each of their elements with another extremal pair in order to link neighboring extremal rays with moving curves.

\subsubsection*{Step 2}
The next step in computing the effective cone is to turn the extremal pairs into resolutions of the general element of the Hilbert scheme.
We will use Thm. \ref{thm:resolution-neg pos} and Thm. \ref{thm:resolution-pos pos} to get these resolutions.
To apply those theorems, we have to complete the pairs to coils as described in Section \ref{sec: Beilinson Spec seq}.
This approach gives the coils
$$\{\OO(-7,-1),\OO(-6,-1),\OO(-7,-0),\OO(-6,0)\}$$ 
$$\{\OO(-7,-2),\OO(-4,-1),\OO(-3,-1),\OO(-6,0)\},$$ 
$$\{\OO(-4,-3),\OO(-4,-2),\OO(-3,-2),\OO(-3,-1)\},\textnormal{ and}$$ 
$$\{\OO(-4,-3),\OO(-3,-2),\OO(-3,-1),\OO(-2,-2)\},$$ 
respectively. 

Given these coils, we get the resolutions we wanted using the spectral sequence as in the proofs of the relevant theorems.
Following the proof, we get the resolutions 
$$0  \to  \OO(-7,-1)^7  \to   \OO(-6,-1)^7 \osum \OO(-7,0)  \to  \mathcal{I}_z  \to  0,$$  
$$0  \to  \OO(-7,-2)  \to  \OO(-4,-1) \osum \OO(-3,-1)  \to  \mathcal{I}_z  \to  0,$$
$$0  \to  \OO(-4,-3) \osum \OO(-4,-2)^2  \to  \OO(-3,-2)^3 \osum \OO(-3,-1)  \to  \mathcal{I}_z  \to  0, \textnormal{ and}$$ 
$$0  \to  \OO(-4,-3) \osum \OO(-3,-2)  \to  \OO(-3,-1) \osum \OO(-2,-2)^2  \to  \mathcal{I}_z  \to  0,$$
respectively. 

\subsubsection*{Step 3}
We now get to the third step in the process, turning the resolutions into maps to Kronecker moduli spaces.
There are no Kronecker modules that are used in the first two resolutions and the Kronecker module in each of the last two resolutions are 
$$\OO(-4,-2)^2  \to  \OO(-3,-2)^3 \textnormal{ and}$$ 
$$\OO(-3,-2)  \to  \OO(-3,-1),$$
respectively. 

This means that we have the maps
$$\pi_1: \left( \QS\right)^{[7]} \dashrightarrow Kr_{\hom(\OO(3,2),\OO(4,2))}(3,2), \textnormal{ and}$$
$$\pi_2: \left( \QS\right)^{[7]} \dashrightarrow Kr_{\hom(\OO(3,1),\OO(3,2))}(1,1),$$
respectively.

Note that the dimensions of these Kronecker moduli spaces are $0$ and $1$, respectively, so we will only consider the map in the last case.

\subsubsection*{Step 4}
The fourth and final step is actually computing the effective divisors and their dual moving curves.
Let $D = aH_1 +bH_2 -c \frac{B}{2}$ be a general effective divisor.

In the first case, the Brill-Noether divisor is $D_V$ where $V = \OO(6,0)$.
The moving curve 
comes from a pencil of maps $\OO(-7,-1)^7 \to \OO(-6,-1)^7$. 
The restriction this moving curve places on $D$ is that $b \geq 0$.
In particular, $B$ and $X_{6,0}$ are dual to this moving curve.

In the second case, the Brill-Noether divisor is $D_V$ where $V = \OO(6,0)$.
The 
moving curve 
comes from a pencil of maps $\OO(-7,-2) \to \OO(-4,-1)$.
The restriction this moving curve places on $D$ is that $3b \geq 6-a$.
In particular, $X_{6,0}$ and $X_{\frac{12}{5},\frac{6}{5}}$ are dual to this moving curve.

For $\pi_1$, the Brill-Noether divisor is $D_V$ where $V$ is the exceptional bundle $E_{\frac{12}{5},\frac{6}{5}}$.
Notice that in this case, the Kronecker fibration is a map to a point.
This implies that the divisor $D_V$ is rigid.
The two types of moving curve come from pencils of maps $\OO(-4,-3) \to  \OO(-3,-2)^3$ and $\OO(-4,-2)^2 \to \OO(-3,-1)$.
These are dual to $D_V$ by the resolution
$$0 \to \OO(3,0)^4 \to E_{\frac{8}{3},\frac{2}{3}}^3 \to V \to 0$$
since $\chi((3,2),(4,3)) = 12+6-2*3*3 = 0$.
The restriction these two moving curves place on $D$ are that $3b \geq 6-a$ and $4b \geq 12-3a$.
In particular, $X_{\frac{12}{5},\frac{6}{5}}$ is dual to both moving curves and $X_{2,\frac{3}{2}}$ is dual to the second moving curve.

For $\pi_2$, the Brill-Noether divisor is $D_V$ where $V$ is a bundle $F_{2,\frac{3}{2}}$ that has Chern character $(2,(4,3),5)$.
The two types of moving curve covering each fiber come from pencils of maps $\OO(-4,-3) \to  K$ and $\OO(-4,-3) \to \OO(-2,-2)^2$.
These are dual to $D_V$ by the resolution
$$0 \to V \to E_{\frac{7}{3},\frac{4}{3}} \to \OO(3,1) \to 0$$
since $\chi((1,1),(1,1)) = 2*1*1-1*1-1*1+1*1 = 0$.
The restriction these two moving curve place on $D$ are that $4b\geq 12-3a$ and $2b\geq 7-2a$.
In particular, $X_{2,\frac{3}{2}}$ is dual to both moving curves, $X_{\frac{12}{5},\frac{6}{5}}$ is dual to the first moving curve, and $X_{\frac{3}{2},2}$ is dual to the second moving curve.

We have now exhibited 7 effective divisors $\left( B, X_{6,0}, X_{\frac{12}{5},\frac{6}{5}},X_{2,\frac{3}{2}},X_{\frac{3}{2},2},X_{\frac{6}{5},\frac{12}{5}},X_{0,6},\right)$ and 7 moving curves that are dual to each pair of extremal rays that span a face of the effective cone.
Taken together, these divisors and moving curves determine the effective cone.


\subsection{Infinite Series of Extremal Rays}
\label{subsec: infinite series of extremal rays}

As another example of the power of the methods produced in this paper, we can construct an extremal ray in the Hilbert scheme of $n$ points for infinite sequences of $n$.
We provide two extremal rays for three such sequences and one extremal ray for a fourth sequence.
The strategy for each proof is to first find an extremal pair, then use the process outlined by our theorems to show that they give the desired extremal ray(s).

The first sequence we look at is actually just all $n$.
For this sequence, we prove what the edges of the effective cone that share the secondary extremal ray are.
\begin{prop}
The edge spanned by $X_{n-1,0}$ and $B$ is an extremal edge of the effective cone of $\left( \QS \right)^{[n]}$.
Similarly, the edge spanned by $X_{n-1,0}$ and $B$ is an extremal edge of the effective cone of $\left( \QS \right)^{[n]}$.
\end{prop}

\begin{proof}
This was proved for $n \leq 5$ in \cite{BC:BC13}.
It is immediate from the symmetry of the effective cone in terms of $a$ and $b$ that the second statement is immediate from the first statement.
We now prove the first statement.

The first step in proving this edge is an extremal ray is finding an extremal pair.
To find an extremal pair, we first have to find the two controlling exceptional bundles that will make up the pair.
The vector bundle whose Brill-Noether divisor will span the ray $X_{n-1,0}$ is $\OO(n-1,0)$.
To find our controlling exceptional bundles, we need to find exceptional bundles cohomologically orthogonal to $\OO(n-1,0)$.
Then,
$$\chi(\OO(n-1,1),\OO(n-1,0))= 0 \textnormal{ and } \chi(\OO(n,0),\OO(n-1,0)) = 0.$$
Then it is easy to see that $\OO(n-1,1)$ and $\OO(n,0)$ are controlling exceptional bundles for the Hilbert scheme, and the pair $\{\OO(n-1,1),\OO(n,0)\}$ is an extremal pair.

Once we have the extremal pair, we need to turn it into a resolution of the general object of the Hilbert scheme.
We complete the pair to a coil as prescribed by Thm. \ref{thm:resolution-pos pos}.
This completion gives the coil
$$\{\OO(-n,-1),E_{\frac{-2-3(n-1)}{3},\frac{-2}{3}},\OO(-n+1,-1),\OO(-n,0) \}.$$
Next, we use the Beilinson spectral sequence to resolve the general ideal sheaf.
The spectral sequence gives the resolution
$$0  \to  \OO(-n,-1)^n  \to  \OO(-n+1,-1)^n \osum \OO(-n,0) \to  \mathcal{I}_z  \to  0.$$


The moving curves 
are pencils in the space $\Hom(\OO(-n,-1),\OO(-n+1,-1))$. 
The restriction this moving curve places on $D$ is that $b \geq 0$.
In particular, $B$ and $X_{n-1,0}$ are dual to this moving curve.

$B$ is known to be an effective divisor.
The ray corresponding to $X_{n-1,0}$ is spanned by the effective Brill-Noether divisor $D_V$ where $V=\OO(n-1,0)$ by Thm. \ref{thm:eff cone neq neq}.
By symmetry, it is clear that $B$ spans an extremal ray.
We have not yet shown that $X_{n-1,0}$ spans an extremal ray because we have only provided one moving curve dual to it.
The next two propositions will complete the proof that it spans an extremal ray by providing a second dual moving curve.
The first proposition provides the dual moving curve in the case that $n$ is even while the second proposition does so in the case that $n$ is odd.
\end{proof}

The next proposition provides another edge of the effective cone in the case that $n$ is even, i.e. $n=2k$.
This edge will share an extremal ray with the edge provided by the previous theorem.
It will provide the second dual moving curve we needed to complete the previous proposition in the case that $n$ is even.

\begin{prop}
The edge spanned by $X_{2k-1,0}$ and $X_{k-1,1}$ is an extremal edge of the effective cone of $\left( \QS \right)^{[2k]}$ for $k>0$.
Similarly, the edge spanned by $X_{0,2k-1}$ and $X_{1,k-1}$ is an extremal edge of the effective cone of $\left( \QS \right)^{[2k]}$ for $k>0$.
\end{prop}

\begin{proof}
This is proved for $k=1$ and $k=2$  in \cite{BC:BC13}.
The proof now proceeds analogously as the previous proof.
Again, it is immediate from the symmetry of the effective cone in terms of $a$ and $b$ that the second statement immediately follows from the first statement.
We now prove the first statement.

The first step in proving this edge is an extremal ray is finding an extremal pair.
To find an extremal pair, we first have to find the two controlling exceptional bundles that will make up the pair.
The vector bundle whose Brill-Noether divisor will span the ray $X_{2k-1,0}$ is $\OO(2k-1,0)$.
The vector bundle whose Brill-Noether divisor will span the ray $X_{k-1,1}$ is $\OO(k-1,1)$.
To find our controlling exceptional bundles, we need to find exceptional bundles cohomologically orthogonal to $\OO(2k-1,0)$ and $\OO(k-1,1)$.
Then, we have that 
$$\chi(\OO(k,1),\OO(2k-1,0))= 0, \textnormal{ } \chi(\OO(2k-1,0),\OO(2k-2,0)) = 0, $$ 
$$\chi(\OO(k,1),\OO(k-1,1))= 0 \textnormal{, and } \chi(\OO(k-1,1),\OO(2k-2,0)) = 0.$$
Next, it is easy to see that $\OO(2k-2,0)$ and $\OO(k,1)$ are controlling exceptional bundles for the Hilbert scheme and that the pair $\{\OO(2k,0),\OO(k,1)\}$ is an extremal pair.

Once we have the extremal pair, we need to turn it into a resolution of the general object of the Hilbert scheme.
We complete the pair to a coil as prescribed by Thm. \ref{thm:resolution-neg pos}.
This completion gives the coil
$$\{\OO(-2k,-2),\OO(-2k+1,-2),\OO(-k-1,-1),\OO(-k,-1) \}.$$
Next, we use the Beilinson spectral sequence to resolve the general ideal sheaf.
The spectral sequence gives the resolution
$$0  \to  \OO(-2k,-2)  \to   \OO(-k,-1)^2  \to  \mathcal{I}_z  \to  0.$$

Using this resolution, the third step is again finding a map to a moduli space of Kronecker modules.
The Kronecker module in this resolution is $\OO(-2k,-2)  \to   \OO(-k,-1)^2$.
Then we get a map 
$$\pi: \left(\QS\right)^{[n]} \dashrightarrow Kr_{\hom(\OO(-2k,-2),\OO(-k,-1))}(1,2).$$  

Using this map, we can finally compute the desired part of the effective cone.
By a straightforward dimension count, we know that this map is birational.
Then the two effective Brill-Noether divisors $D_V$ and $D_{V'}$ where $V = \OO(2k-1,1)$ and $V' = \OO(k-1,1)$ are contracted by this map.
Next, any moving curve in the Kronecker moduli space is dual to these contracted divisors.
Thus, a pencil in the space $\Hom(\OO(-2k,-2),\OO(-k,-1))$ provides a dual moving curve showing that these divisors are on an edge of the effective cone.
Alternatively, we could show that this moving curve gives the restriction $kb \geq 2k-1-a$.
Coupled with the previous proposition, it is clear that $X_{2k-1,0}$ spans an extremal ray.

In order to show that $X_{k-1,1}$ is an extremal ray at the other end of the edge, we have to provide another extremal pair.
The extremal pair needed is $\{\OO(k-2,1), \OO(k-1,1)\}$
Then we get the coil
$$\{\OO(-k,-3),\OO(-k,-2),\OO(-k+1,-2),\OO(-k+1,-1) \}.$$
%
The spectral sequence gives the resolution
$$0  \to  \OO(-k,-3)^2 \osum \OO(-k,-2)^{k-3}  \to  \OO(-k+1,-2)^k  \to  \mathcal{I}_z  \to  0.$$

A moving curve 
is a pencil in the space $\Hom(\OO(-k,-2),\OO(-k+1,-2))$.
The restriction this moving curve places on $D$ is that $kb \geq 4k-3-3a $.
In particular, $X_{k-1,1}$ is dual to this moving curve which has a different slope than the other moving curve we constructed through this divisor, so we have shown that it is an extremal ray as promised.
\end{proof}

We now move on to the analogous proposition for odd $n$.
\begin{prop}
The edge spanned by $X_{2k,0}$ and $X_{\frac{2k(k-1)}{2k-1},\frac{2k}{2k-1}}$ is an extremal edge of the effective cone of $\left( \QS \right)^{[2k+1]}$ for $k>1$.
Similarly, the edge spanned by $X_{0,2k}$ and $X_{\frac{2k}{2k-1},\frac{2k(k-1)}{2k-1}}$ is an extremal edge of the effective cone of $\left( \QS \right)^{[2k+1]}$ for $k>1$.
\end{prop}

\begin{proof}
This is shown for $k=2$ in \cite{BC:BC13}.
Assume $k>2$. 
This proof proceeds with all of the same elements as the previous proof, but slightly altered notation due to $n$ being odd.
Due to this, we give a much briefer proof.

The first statement implies the second statement by the symmetry of the effective cone so we prove only the first statement.
Then it can be shown that the pair $\{\OO(2k-1,0),\OO(k,1)\}$ is an extremal pair.
Then we get the coil
$$\{\OO(-2k-1,-2),\OO(-2k,-2),\OO(-k-1,-1),\OO(-k,-1) \}.$$
The spectral sequence gives the resolution
$$0  \to  \OO(-2k-1,-2)  \to  \OO(-k-1,-1) \osum \OO(-k,-1)  \to  \mathcal{I}_z  \to  0.$$

A moving curve 
is a pencil in the space $\Hom(\OO(-2k-1,-2),\OO(-k,-1))$.
The restriction this moving curve places on $D$ is that $kb \geq 2k-a $.
In particular, $X_{2k,0}$ and $X_{\frac{2k(k-1)}{2k-1},\frac{2k}{2k-1}}$ are dual to this moving curve.

Then the two effective Brill-Noether divisors $D_V$ and $D_{V'}$ where $V$ is the exceptional bundle $\OO(2k,0)$ and $V'$ is the exceptional bundle $E_{\frac{2k(k-1)}{2k-1},\frac{2k}{2k-1}}$ are shown to be on an edge by this moving curve.
Coupled with the previous proposition, it is clear that $X_{2k-1,0}$ spans an extremal ray.

In order to show that $X_{\frac{2k(k-1)}{2k-1},\frac{2k}{2k-1}}$ is an extremal ray at the other end of the edge, we have to provide another extremal pair.
The extremal pair needed is $\{\OO(k-1,1), \OO(k,1)\}$
Then we get the coil
$$\{\OO(-k-1,-3),\OO(-k-1,-2),\OO(-k,-2),\OO(-k,-1) \}.$$
%
The spectral sequence gives the resolution
$$0  \to  \OO(-k-1,-3) \osum \OO(-k-1,-2)^{k-1}  \to  \OO(-k,-2)^k \osum \OO(-k,-1)  \to  \mathcal{I}_z  \to  0.$$
Then we get a map 
$$\pi: \left(\QS\right)^{[n]} \dashrightarrow Kr_{\hom(\OO(-k-1,-2),\OO(-k,-2))}(k-1,k).$$  

By a dimension count, we see that this Kronecker moduli space is zero dimensional so we disregard it.
The moving curves 
are pencils in the spaces $\Hom(\OO(-k-1,-3),\OO(-k,-2))$ and $\Hom(\OO(-k-1,-2),\OO(-k,-1))$.
The restrictions these moving curves place on $D$ are that $kb \geq 2k-a$ and $(k+1)b \geq 4k-3a$.
In particular, $X_{\frac{2k(k-1)}{2k-1},\frac{2k}{2k-1}}$ is dual to these moving curve which have different slopes, so we have shown that it is an extremal ray as promised.
\end{proof}

The final sequence we look at is $n =3k+1$.
We provide this sequence as an example of the large class of extremal rays that be found in more sporadic sequences.
\begin{prop}
The ray spanned by $X_{k-\frac{1}{2},2}$ is an extremal ray of the effective cone of $\left( \QS \right)^{[3k+1]}$ for $k>1$.
Similarly, the ray spanned by $X_{2,k-\frac{1}{2}}$ is an extremal ray of the effective cone of $\left( \QS \right)^{[3k+1]}$ for $k>1$.
\end{prop}

\begin{proof}
This proof again proceeds similarly to the previous proofs so we provide a concise version.
Again, the second statement follows from the first by symmetry.
The extremal controlling pair is $\{\OO(k-1,2),\OO(k,2)\}$
This completes to the coil $$\{\OO(-k-1,-4),\OO(-k,-3), \OO(-k+1,-3),\OO(-k,-2)\}.$$
Then the resolution that we get is 
$$0  \to  \OO(-k-1,-4)  \osum \OO(-k,-3)^{k-1} \to  \OO(-k+1,-3)^{k-1} \osum \OO(-k,-2)^2  \to  I_z  \to  0.$$
Then we get a map 
$$\pi: \left(\QS\right)^{[n]} \dashrightarrow Kr_{\hom(\OO(-k,-3),\OO(-k+1,-3))}(k-1,k-1),$$  
and $V$ has the resolution 
$$0 \to  \OO(k,1) \to E_{\frac{k-1}{3},\frac{7}{3}} \to V \to 0.$$
Next, the Kronecker modules are dual, so we have cohomological orthogonality.
This makes $D_V$ into a divisor.
By Prop. \ref{prop:brillnoether}, we know its class is $X_{k-\frac{1}{2},2}$.
We see that it is an extremal ray by looking at pencils in the spaces $\Hom(\OO(-k-1,-4),\OO(-k+1,-3))$ and $\Hom(\OO(-k,-3),\OO(-k,-2))$ which cover the fibers of the Kronecker fibration.
The restrictions these moving curve places on $D$ are $(1+k)b \geq 6k-4a $ and $kb \geq 4k-1-2a$.
Note, this implies they are distinct curve classes and $X_{k-\frac{1}{2},2}$ is dual to both moving curves.
Thus, the Brill-Noether divisor $D_V$ where $V$ is a bundle with Chern character $\left(2,(2k-1,4),4k-3 \right)$ spans an extremal ray of the effective cone.
\end{proof}

There are a couple more infinite families that we will mention but not prove.
Their proofs follow similar techniques. 
Working out their proofs is a good exercise to become comfortable with this type of computation.  

These families of rays on a edge require some notation.
\begin{definition}
The $\textbf{symmetric value}$ of the effective cone of $\left( \QS \right)^{[N]}$ is the value $a$ for which the ray spanned by $X_{a,a}$ is on the edge of the effective cone.
\end{definition}

Note $X_{a,a}$ may or may not be an extremal ray.
We now state the symmetric value for four infinite sequences of $n$.
\begin{prop}
The symmetric value of the effective cone of $\left( \QS \right)^{[n]}$ is:

I)   $\hspace{.105in}k-1 - \frac{1}{2k-2} \hspace{.49in}$   for $n = k^2-2$, $k>1$

II)  $\hspace{.06in}k-1 \hspace{.95in}$                    for $n = k^2-1 \textnormal{ or } k^2$, $k>1$

III) $k-1+\frac{1}{k+1} \hspace{.55in}$      for $n = k^2+1$, $k>1$

IV)  $k-\frac{1}{2}\hspace{.94in}$          for $n = k^2+k$, $k>0$
\end{prop}

\subsection{Completing the Table}
Finally, using our methods, we will give brief proofs of each of the five corners in the table at the beginning of this section that do not follow from our general constructions so far.
We will only state the propositions and proofs for one of each pair of symmetric extremal rays.

\begin{prop}
$X_{4,\frac{4}{3}}$ is an extremal ray of $\left( \QS\right)^{[11]}$.
\end{prop}

\begin{proof}
Let $Z \in \left( \QS\right)^{[11]}$ be general.
The relevant extremal pair is $\{\OO(-4,-1),\OO(-3,-2)\}$.
The resolving coil then will be $$\{\OO(-6,-3),\OO(-4,-2),E_{\frac{-11}{3},\frac{-5}{3}},\OO(-3,-2)\}$$ since 
$$\chi(\OO(-4,-1),\mathcal{I}_Z) =  1*1((1+0+4)(1+0+1)-0-11) = -1 \textnormal{, and}$$
$$\chi(\OO(-3,-2),\mathcal{I}_Z) = 1*1((1+0+3)(1+0+2)-0-11) = 1.$$
Then the resolutions we get from the generalized Beilinson spectral sequence are 
$$0 \to \OO(-6,-3) \osum \OO(-4,-2)^2 \to E_{\frac{-11}{3},\frac{-5}{3}} \osum \OO(-3,-2) \to \mathcal{I}_Z \to 0 \textnormal{ and }$$
$$0 \to \OO(2,2) \to V \to \OO(5,1)^2 \to 0.$$
Next, the Kronecker map is
$$\pi: \left(\QS\right)^{[11]}  \dashrightarrow Kr_{\hom\left(E_{\frac{11}{3},\frac{5}{3}},\OO(4,2)\right)}(1,2).$$ 
Note that the dimension of this Kronecker moduli space is $4*2*1-2^2-1^2+1 = 4$

Then, the Brill-Noether divisor is $D_V$ where $V$ is a bundle that has Chern character $(3,(12,4),14)$.
The two types of moving curves covering each fiber come from pencils of maps $\OO(-6,-3) \to K$ 
and $\OO(-6,-3) \to \OO(-3,-2)$.
The restrictions these two moving curves place on $D$ are that $3b\geq 12-2a$ and $6b\geq 20-3a$.
In particular, $X_{4,\frac{4}{3}}$ is dual to both moving curves, $X_{\frac{40}{9},\frac{10}{9}}$ is dual to the first moving curve, and $X_{\frac{12}{5},\frac{12}{5}}$ is dual to the second moving curve.
\end{proof}

\begin{prop}
$X_{\frac{12}{5},\frac{12}{5}}$ is an extremal ray of $\left( \QS\right)^{[11]}$.
\end{prop}

\begin{proof}
Let $Z \in \left( \QS\right)^{[11]}$ be general.
The relevant extremal pair is $\{\OO(-2,-3),\OO(-3,-2)\}$.
The resolving coil then will be $$\{\OO(-4,-4),\OO(-3,-3),\OO(-2,-3),\OO(-3,-2)\}$$ since 
$$\chi(\OO(-2,-3),\mathcal{I}_Z) =  1*1((1+0+2)(1+0+3)-0-11) = 1 \textnormal{, and}$$
$$\chi(\OO(-3,-2),\mathcal{I}_Z) = 1*1((1+0+3)(1+0+2)-0-11) = 1.$$
Then the resolutions we get from the generalized Beilinson spectral sequence are
$$0 \to \OO(-4,-4)^2 \to \OO(-3,-3) \osum \OO(-2,-3) \osum \OO(-3,-2) \to \mathcal{I}_Z \to 0 \textnormal{ and }$$
$$0 \to E_{\frac{7}{3},\frac{7}{3}}^2 \to E_{\frac{26}{11},\frac{26}{11}} \to V \to 0.$$
Next, the Kronecker map is
$$\pi: \left( \QS\right)^{[11]} \dashrightarrow Kr_{\hom\left(\OO(3,3),\OO(4,4)\right)}(1,2).$$ 
Note that the dimension of this Kronecker moduli space is $2*2*1-2^2-1^2+1 = 0$ so the map tells us nothing.

Then, the Brill-Noether divisor is $D_V$ where $V$ is a bundle that has Chern character $(5,(12,12),26)$.
There are two types of moving curves coming from pencils of maps $K \to \OO(-2,-3)$ and $K \to \OO(-3,-2)$.
The restrictions these two moving curves place on $D$ are that $3b\geq 20-6a$ and $6b\geq 20-3a$.
In particular, $X_{\frac{12}{5},\frac{12}{5}}$ is dual to both moving curves, $X_{4,\frac{4}{3}}$ is dual to the first moving curve, and $X_{\frac{4}{3},4}$ is dual to the second moving curve.
\end{proof}

\begin{prop}
$X_{\frac{9}{2},\frac{3}{2}}$ is an extremal ray of $\left( \QS\right)^{[13]}$.
\end{prop}

\begin{proof}
Let $Z \in \left( \QS\right)^{[13]}$ be general.
The relevant extremal pair is $\{\OO(-5,-1),\OO(-4,-2)\}$.
The resolving coil then will be $$\{\OO(-7,-3),\OO(-5,-2),E_{\frac{-14}{3},\frac{-5}{3}},\OO(-4,-2)\}$$ since 
$$\chi(\OO(-5,-1),\mathcal{I}_Z) =  1*1((1+0+5)(1+0+1)-0-13) = -1 \textnormal{, and}$$
$$\chi(\OO(-4,-2),\mathcal{I}_Z) = 1*1((1+0+4)(1+0+2)-0-13) = 2.$$
Then the resolutions we get from the generalized Beilinson spectral sequence are
$$0 \to \OO(-7,-3) \osum \OO(-5,-2)^3 \to E_{\frac{-14}{3},\frac{-5}{3}} \osum \OO(-4,-2)^2 \to \mathcal{I}_Z \to 0 \textnormal{ and }$$
$$0 \to \OO(6,1) \to V \to \OO(3,2) \to 0.$$
Next, the Kronecker map is
$$\pi: \left( \QS\right)^{[13]}  \dashrightarrow Kr_{\hom\left(E_{\frac{14}{3},\frac{5}{3}},\OO(5,2)\right)}(1,3).$$ 
Note that the dimension of this Kronecker moduli space is $4*1*3-3^2-1^2+1 = 3$.

Then, the Brill-Noether divisor is $D_V$ where $V$ is a bundle that has Chern character $(2,(9,3),12)$.
There are two types of moving curves coming from pencils of maps $\OO(-7,-3) \to K$ 
and $\OO(-7,-3) \to  \OO(-4,-2)^2$.
The restrictions these two moving curves place on $D$ are that $4b\geq 15-2a$ and $7b\geq 24-3a$.
In particular, $X_{\frac{9}{2},\frac{3}{2}}$ is dual to both moving curves, $X_{\frac{7}{2},2}$ is dual to the first moving curve, and $X_{\frac{60}{11},\frac{12}{11}}$ is dual to the second moving curve.
\end{proof}

\begin{prop}
$X_{\frac{8}{3},\frac{8}{3}}$ is an extremal ray of $\left( \QS\right)^{[13]}$.
\end{prop}

\begin{proof}
Let $Z \in \left( \QS\right)^{[13]}$ be general.
The relevant extremal pair is $\{\OO(-2,-3),\OO(-3,-2)\}$.
The resolving coil then will be $$\{\OO(-4,-5),\OO(-5,-4),\OO(-4,-4),\OO(-3,-3)\}$$ since 
$$\chi(\OO(-2,-3),\mathcal{I}_Z) =  1*1((1+0+2)(1+0+3)-0-13) = -1 \textnormal{, and}$$
$$\chi(\OO(-3,-2),\mathcal{I}_Z) = 1*1((1+0+3)(1+0+2)-0-13) = -1.$$
Then the resolution we get from the generalized Beilinson spectral sequence is 
$$0 \to \OO(-4,-5) \osum \OO(-5,-4) \to \OO(-3,-3)^3 \to \mathcal{I}_Z \to 0,$$ 
so this is no Kronecker map.

Next, the Brill-Noether divisor is $D_V$ where $V$ is the exceptional bundle $E_{\frac{8}{3},\frac{8}{3}}.$ 
There are two types of moving curves coming from pencils of maps $\OO(-4,-5) \to \OO(-3,-3)^3$ and $\OO(-5,-4) \to  \OO(-3,-3)^3$.
The restrictions these two moving curves place on $D$ are that $5b\geq 24-4a$ and $4b\geq 24-5a$.
In particular, $X_{\frac{8}{3},\frac{8}{3}}$ is dual to both moving curves, $X_{\frac{7}{2},2}$ is dual to the first moving curve, and $X_{2,\frac{7}{2}}$ is dual to the second moving curve.
\end{proof}

\begin{prop}
$X_{\frac{10}{3},\frac{7}{3}}$ is an extremal ray of $\left( \QS\right)^{[14]}$.
\end{prop}

\begin{proof}
Let $Z \in \left( \QS\right)^{[14]}$ be general.
The relevant extremal pair is $\{\OO(-3,-3),\OO(-4,-2)\}$.
The resolving coil then will be $$\{\OO(-5,-4),\OO(-4,-3),\OO(-4,-2),\OO(-3,-3)\}$$ since 
$$\chi(\OO(-3,-3),\mathcal{I}_Z) =  1*1((1+0+3)(1+0+3)-0-14) = 2 \textnormal{, and}$$
$$\chi(\OO(-4,-2),\mathcal{I}_Z) = 1*1((1+0+4)(1+0+2)-0-14) = 1.$$
Then the resolution we get from the generalized Beilinson spectral sequence is 
$$0 \to \OO(-5,-4)^2 \to \OO(-4,-2) \osum \OO(-3,-3)^2 \to \mathcal{I}_Z \to 0,$$
so this is no Kronecker map.

Next, the Brill-Noether divisor is $D_V$ where $V$ is the exceptional bundle $E_{\frac{10}{3},\frac{7}{3}}.$ 
There are two types of moving curves coming from pencils of maps $\OO(-5,-4)^2 \to \OO(-4,-2)$ and $\OO(-5,-4)^2 \to  \OO(-3,-3)^2$.
The restriction these two moving curves place on $D$ are that $3b\geq 17-3a$ and $4b\geq 16-2a$.
In particular, $X_{\frac{10}{3},\frac{7}{3}}$ is dual to both moving curves, $X_{\frac{7}{3},\frac{10}{3}}$ is dual to the first moving curve, and $X_{6,1}$ is dual to the second moving curve.
\end{proof}

\subsection{A Rank Two Example}
Let $\xi = \left(\textnormal{log}(2),(\frac{1}{2},0),2 \right)$
Then, we can find that $\{\OO(-1,2),\OO(0,1)\}$ is an extremal pair for $M(\xi)$.
It gives the resolution 
$$0 \to \OO(-1,-4) \to \OO(0,-2) \osum \OO(0,-1)^2 \to U \to 0$$
where $U\in M(\xi)$ is general.
This gives that the divisor $D_{\OO(-1,3)}$ lies on an edge of the effective cone.
In this case, we can actually do two dimension counts to see that varying either map gives a moving curve dual to $D_{\OO(-1,3)}$ so it in fact spans an extremal ray of the effective cone.

\bibliographystyle{alphanum}
\bibliography{fullbibliography}
\end{document}